\documentclass[11pt]{amsart}

\usepackage{amssymb,amsmath}
\usepackage{enumitem}
\usepackage{graphicx}
\usepackage{verbatim}
\usepackage{multirow}
\usepackage{framed}
\usepackage{color}

\newenvironment{thmenum}{\begin{enumerate}[leftmargin=2em, itemsep=.5ex, label=\textbf{\Roman{*}.}]}{\end{enumerate}}

\newcommand{\foot}[1]{\mbox{}\marginpar{\raggedleft\hspace{0pt}\tiny #1}}
\newcommand{\eps}{\varepsilon}
\newcommand{\ph}{\varphi}
\newcommand{\NN}{\mathbb{N}}
\newcommand{\ZZ}{\mathbb{Z}}
\newcommand{\RR}{\mathbb{R}}
\newcommand{\one}{\mathbf{1}}
\newcommand{\AAA}{\mathcal{A}}
\newcommand{\CCC}{\mathcal{C}}
\newcommand{\MMM}{\mathcal{M}}
\newcommand{\WWW}{\mathcal{W}}

\newcommand{\GGG}{\mathcal{G}}

\newcommand{\llim}{\varliminf}
\newcommand{\ulim}{\varlimsup}

\newcommand{\hl}{{\hat{\lambda}}}
\newcommand{\cl}{{\check{\lambda}}}
\newcommand{\hZ}{{\hat{Z}}}

\newcommand{\ba}{{\overline{\theta}}}

\newcommand{\hA}{\hat{A}}

\newcommand{\di}{\partial}
\newcommand{\ld}{\underline{\delta}}
\newcommand{\ud}{\overline{\delta}}
\newcommand{\nnorm}[1]{\left\lVert #1 \right\rVert_*}
\newcommand{\ex}[1]{\exp\left(#1\right)}
\newcommand{\chigap}{\chi^g}
\newcommand{\Cp}{\textbf{(C$^\prime$)}}

\DeclareMathOperator{\graph}{graph}

\DeclareMathOperator{\Leb}{Leb}

\theoremstyle{plain}
\newtheorem{theorem}{Theorem}[section]
\newtheorem{proposition}[theorem]{Proposition}

\newtheorem{lemma}[theorem]{Lemma}
\newtheorem{thma}{Theorem}

\theoremstyle{definition}
\newtheorem{definition}{Definition}[section]

\theoremstyle{remark}

\newtheorem{remark}[theorem]{Remark}

\numberwithin{equation}{section}

\title
{Hadamard--Perron theorems and effective hyperbolicity}
\author{Vaughn Climenhaga}
\address{Department of Mathematics \\ University of Houston \\ Houston, TX 77204, USA}
\email{climenha@math.uh.edu}
\urladdr{http://www.math.uh.edu/$\sim$climenha/}

\author{Yakov Pesin}
\address{Department of Mathematics \\ McAllister Building \\ Pennsylvania State University \\ University Park, PA 16802, USA}
\email{pesin@math.psu.edu}
\urladdr{http://www.math.psu.edu/pesin/}

\begin{document}

\date{\today}
\begin{abstract}
We prove several new versions of the Hadamard--Perron Theorem, which relates infinitesimal dynamics to local dynamics for a sequence of local diffeomorphisms, and in particular establishes the existence of local stable and unstable manifolds.  Our results imply the classical Hadamard--Perron Theorem in both its uniform and non-uniform versions, but also apply much more generally.  We introduce a notion of ``effective hyperbolicity'' and show that if the rate of effective hyperbolicity is asymptotically positive, then the local manifolds are well-behaved with positive asymptotic frequency.  By applying effective hyperbolicity to finite orbit segments, we prove a closing lemma whose conditions can be verified with a finite amount of information.
\end{abstract}

\thanks{The authors were partially supported by NSF grant 0754911. Ya.\ P.\ is partially supported by NSF grant 1101165. V.C.\ was supported by an NSERC postdoctoral fellowship.}

\maketitle

\section{Introduction}

\begin{quotation}
\emph{Every five years or so, if not more often, someone ``discovers'' the theorem of Hadamard and Perron, proving it either by Hadamard's method of proof or by Perron's.  I myself have been guilty of this.}\\
\parbox{1cm}{} \hfill D.V.\ Anosov, 1967. \cite[p.\ 23]{dA67}\\[1ex]
\end{quotation}

Following in the footsteps of Anosov and many others, we prove several new versions 
of the Hadamard--Perron theorem on the construction  of local stable and unstable manifolds (taking our inspiration from Hadamard's method of proof).  This theorem in its various incarnations is one of the key tools in the theory of hyperbolic dynamical systems, both uniform and non-uniform.  Informally, it may be thought of as the bridge between the dynamics of the derivative cocycle in the tangent bundle and the dynamics of the original map on the manifold itself.

Although the theorem is primarily used to study a diffeomorphism $f$ on some Riemannian manifold $\mathcal{M}$, it is typically stated in terms of a sequence of germs of diffeomorphisms.  That is, one fixes an initial point $x\in \mathcal{M}$ and then writes $f_n$ for the restriction of the map $f$ to a neighbourhood $\Omega_n$ of $f^n(x)$.  Using local coordinates from $T_{f^n(x)} \mathcal{M}$, we can  view $\Omega_n$ as a neighbourhood in $\RR^d$ and write $f_n\colon \Omega_n \to \RR^d$, where $d=\dim \mathcal{M}$.

Roughly speaking, the content of the Hadamard--Perron theorem is as follows: if there is an invariant splitting $\RR^d = E_n^u \oplus E_n^s$ and $\lambda < 1$ such that $\|Df_n(0)|_{E_n^s}\| < \lambda < \|Df_n(0)|_{E_n^u}^{-1}\|^{-1}$ for every $n$, then under some additional assumptions on $f_n$ there are uniquely defined local stable manifolds $W_n^s\ni 0$ tangent to $E_n^s$ at $0$ such that $d(f_n(x),f_n(y)) \leq \lambda d(x,y)$ for every $x,y\in W_n^u$.  Moreover, if $V_n$ is any \emph{admissible manifold} transverse to $E_n^u$ at $0$, then the sequence of admissible manifolds $f^{-k}(V_n)$ converges to the stable manifolds $W_{n-k}^u$ as $k\to\infty$.

Within this general framework, various versions of the theorem have been stated in which the precise hypotheses and conclusions vary. In these versions one usually works with stable manifolds, as described above; the local unstable manifolds are then obtained as being stable for the sequence of inverse maps $f^{-1}_n$. We stress that for some technical reasons and in view of some applications of our results (see Section \ref{sec:srb}) we will construct local unstable manifolds first. 

In Section~\ref{sec:prelim}, we describe how the present paper fits into previous results and give the precise setting and notation in which we will work.

In Section~\ref{sec:C1+}, we give results applying to sequences of $C^{1+\alpha}$ maps. We introduce the notion of \emph{effective hyperbolicity}, and show that for an effectively hyperbolic sequence of $C^{1+\alpha}$ diffeomorphisms 
$\{f_n\mid n\geq 0\}$, one can control non-uniformities in the admissible manifolds and their associated dynamics.  Our main result is Theorem~\ref{thm:HP2}, a new version of the Hadamard--Perron theorem that deals with pushing forward an admissible manifold under the maps $f_n$.  While the images may not have good properties for all $n$, they do have good properties on the set of \emph{effective hyperbolic times}, which has positive asymptotic frequency provided the sequence of maps is effectively hyperbolic.


While Theorem \ref{thm:HP2} is of interest in its own right, it is also used in our companion paper~\cite{CDP11b} to construct SRB measures for general non-uniformly hyperbolic attractors; a description is given in Section~\ref{sec:srb} (in particular, see Theorem~\ref{srb-measure}).  Effective hyperbolicity can be established in situations where the system has good recurrence properties to a part of the  phase space with uniformly hyperbolic behaviour, and where we have some control on the behaviour of the map when the trajectory leaves this region.

In Theorem~\ref{thm:HP5}, we use effective hyperbolicity to give criteria for the existence and uniqueness of local unstable manifolds for a sequence of $C^{1+\alpha}$ diffeomorphisms $\{f_n \mid n\leq 0\}$.  Morally speaking, Theorems~\ref{thm:HP2} and~\ref{thm:HP5}, and to some degree this entire paper, can be summed up as follows (definitions of the three properties below can be found in~\eqref{eqn:lambdabig}, \eqref{eqn:eft}, and~\eqref{eqn:asymp-dom}, respectively):

\begin{center}
\begin{tabular}{rcl}
\multirow{2}{*}{\textbf{effective hyperbolicity}} & \multirow{2}{*}{$\Rightarrow$} & \textbf{existence of local unstable} \\ & & \textbf{(stable) manifolds} \\[.5ex]
\multirow{2}{*}{\textbf{effective hyperbolic times}} & \multirow{2}{*}{$\Rightarrow$} & \textbf{uniform  bounds on dynamics and} \\ & & \textbf{geometry of admissible manifolds} \\[.5ex]
\multirow{2}{*}{\textbf{asymptotic domination}} & \multirow{2}{*}{$\Rightarrow$} &\textbf{uniqueness of local unstable} \\ & & \textbf{(stable) manifolds}
\end{tabular}
\end{center}


Our strongest result for $C^{1+\alpha}$ maps is Theorem~\ref{thm:parameters}, which gives more precise (and more technical) bounds on the images of admissible manifolds under the graph transform; these are used in the proofs of Theorems~\ref{thm:HP2} and~\ref{thm:HP5}.

The bounds in Theorem~\ref{thm:parameters} depend on two things:
\begin{enumerate}[label=(\roman{*})]
\item\label{lin} linear information on dynamics (controlling contraction and expansion rates of $Df_n$);
\item\label{nonlin} non-linear bounds on dynamics (controlling the modulus of continuity of $Df_n$) and non-uniformities in geometry (controlling the angle between the directions of contraction and expansion).
\end{enumerate}
Using effective hyperbolicity, we can obtain bounds that depend only on the linear information in \ref{lin} and the frequency with which the quantities in \ref{nonlin} exceed certain thresholds (see~\eqref{eqn:Mnleq} and Section~\ref{sec:verifying}).  This is done in Theorem~\ref{thm:tau}.

In Sections~\ref{sec:C1+mfd}--\ref{sec:closing}, we give some principal applications of our results to diffeomorphisms of compact manifolds.
First, in Section \ref{sec:C1+mfd} we introduce the concept of effective hyperbolicity and establish existence of stable and unstable local manifolds along effectively hyperbolic trajectories. In Section \ref{sec:srb} we show how our results can be used to establish existence of Sinai--Ruelle--Bowen (SRB) measures for a broad class of diffeomorphisms that are effectively hyperbolic on a set of positive volume.  Finally, in Section \ref{sec:closing} we prove an adaptation of the classical closing lemma to effectively hyperbolic diffeomorphisms.


Sections \ref{sec:C1}--\ref{sec:appspfs} contain the proofs.  The key tool is Theorem~\ref{thm:HP1}, which is a strengthened (and rather more technical) version of Theorem~\ref{thm:parameters} for $C^1$ maps.  Theorem~\ref{thm:HP1} leads to a result on unstable manifolds in Theorem~\ref{thm:HP3}, which is used in the proof of Theorem~\ref{thm:HP5}.  

Following the proofs of the main results, in Section~\ref{sec:classical} we show that Theorem~\ref{thm:HP3} can be used to prove the classical uniform and non-uniform Hadamard--Perron theorems for $C^1$ and $C^{1+\alpha}$ diffeomorphisms, respectively (see Theorems \ref{thm:uniform} and \ref{thm:nonuniform}), and in Section~\ref{sec:nuh} we give some examples illustrating the relationship between effective hyperbolicity and classical notions of non-uniform hyperbolicity.

The following table shows the overall logical structure of our main results and applications.
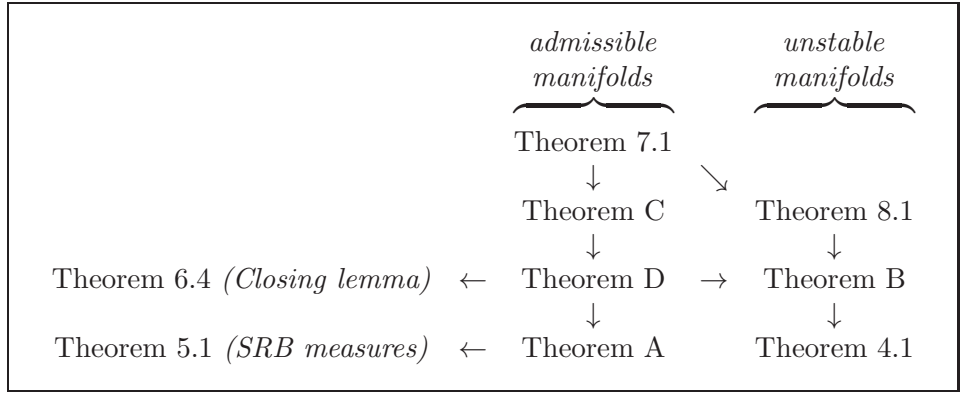
\begin{center}
\begin{framed}
\begin{tabular}{ c c c c c}
&& \emph{admissible} && \emph{unstable} \\
&& \emph{manifolds} && \emph{manifolds} \\
&& $\overbrace{\phantom{\text{Theorem 7.1}}}$ && $\overbrace{\phantom{\text{Theorem 8.1}}}$\\[-1.5ex]
&& Theorem \ref{thm:HP1} \\
&& $\downarrow$ & $\searrow$ \\
&& Theorem \ref{thm:parameters} && Theorem \ref{thm:HP3} \\
&& $\downarrow$ && $\downarrow$ \\
Theorem \ref{thm:closing} \emph{(Closing lemma)} & $\leftarrow$ & Theorem  \ref{thm:tau} & $\rightarrow$ & Theorem \ref{thm:HP5} \\
&& $\downarrow$ && $\downarrow$ \\
Theorem \ref{srb-measure} \emph{(SRB measures)} &$\leftarrow$ & Theorem \ref{thm:HP2} && Theorem \ref{thm:HP6}
\end{tabular}
\end{framed}
\end{center}


\subsection*{Acknowledgments.}   This paper had its genesis as part of a larger joint work with Dmitry Dolgopyat, to whom we are grateful for many helpful discussions and insights.   Part of this research was carried out while the authors were visiting The Fields Institute.

\section{Preliminaries}\label{sec:prelim}

\subsection{Notation and general setting}\label{sec:notation}

Given $n\in\ZZ$, write $V_n = \RR^d$. Let $\Omega_n \subset V_n$ be an open set containing the origin, and $f_n\colon \Omega_n \to V_{n+1}$ a sequence of maps.\footnote{Each $V_n$ is identical to all the others, but we use this notation to make it easier to keep track of the domain and range of various compositions of the maps $f_n$.}  We make the following standing assumptions.\footnote{Although these are formulated for all $n \in \ZZ$, we will in fact mostly be interested in situations where it is appropriate to consider only some subset of $\ZZ$ -- see Remark \ref{rmk:differentn}.}
\begin{enumerate}[leftmargin=3em, label=\textbf{(C\arabic{*})}]
\item\label{C1}
Each $f_n$ is a $C^{1+\alpha}$ diffeomorphism onto its image for some $\alpha\in (0,1]$ (independent of $n$), and $f_n(0)=0$.\footnote{In the proofs, we will treat the more general (but technically messier) $C^1$ case where $Df_n$ have moduli of continuity that are not necessarily H\"older.}
\item\label{C2}
There is a decomposition $V_n = E_n^u \oplus E_n^s$, which is invariant under $Df_n(0)$ -- that is, $Df_n(0)E_n^\sigma = E_{n+1}^\sigma$ for $\sigma=s,u$.
\item\label{C3}
There are numbers $\lambda_n^u,\lambda_n^s \in \RR$ and $\theta_n,\beta_n>0$ such that for every $v_u\in E_n^u$ and $v_s\in E_n^s$, we have
\begin{align}
\label{eqn:C3a}
\|Df_n(0)(v_u)\| &\geq e^{\lambda_n^u} \|v_u\|, \\
\label{eqn:C3b}
\|Df_n(0)(v_s)\| &\leq e^{\lambda_n^s} \|v_s\|, \\
\label{eqn:C3c}
\measuredangle(v_u,v_s) &\geq \theta_n, \\
\label{eqn:C3d}
\max(1,|Df_n|_\alpha) &\leq \beta_n \sin\theta_{n+1},
\end{align}
where $|Df_n|_\alpha$ is the H\"older semi-norm of $Df_n$ (defined in~\eqref{eqn:norm}).
\item\label{C4}
There is $L>0$ such that $|\lambda_n^u| \leq L$, $|\lambda_n^s|\leq L$, and $\beta_{n+1} \leq e^L \beta_n$.
\end{enumerate}

\begin{remark}
Condition~\ref{C2} can be trivially satisfied by fixing any decomposition $V_0 = E_0^u \oplus E_0^s$ and iterating it under $Df_n(0)$.  However, the point is that the angle between $E_n^s$ and $E_n^u$ needs to be controlled by $\theta_n$ as in~\eqref{eqn:C3c}, and our main results will require some control of $\theta_n$.  More generally, we remark that the purpose of Condition~\ref{C3} is to control the dynamics of $Df_n$ with respect to the invariant decomposition $V_n = E_n^u \oplus E_n^s$.
\end{remark}

\begin{remark}\label{rmk:cones}
In applications, it is often more convenient to work with invariant cone families rather than subspaces -- that is, given $E_n^\sigma\subset V_n$ and $\zeta_n^\sigma>0$ ($\sigma=s,u$), one may consider the cones $K_n^\sigma = \{v\in V_n \mid \measuredangle(v,E_n^\sigma) < \zeta_n\}$ and then replace \ref{C2} and \ref{C3} with the following conditions.
\begin{enumerate}[leftmargin=3em, label=\textbf{(C\arabic{*}$^\ast$)}]
\setcounter{enumi}{1}
\item\label{C2*} There is a (not necessarily invariant) decomposition $V_n = E_n^u \oplus E_n^s$ and cone families $K_n^{u,s}$ around $E_n^{u,s}$ such that $\overline{Df_n(0)(K_n^u)} \subset K_{n+1}^u$ and $\overline{Df_n(0)^{-1}(K_{n+1}^s)} \subset K_n^s$.
\item \label{C3*} The bounds in \ref{C3} hold for all $v^u\in K_n^u$ and $v^s\in K_n^s$.
\end{enumerate}
Given a cone family satisfying \ref{C2*} and \ref{C3*}, one can derive splittings  $E_n^u\oplus E_n^s$ satisfying \ref{C2} and \ref{C3}.  For the stable direction, take $E_n^s$ to be any subspace (of the appropriate dimension) in the intersection $\tilde K_n^s = \bigcap_{m\geq 0}Df_{n+1}(0)^{-1} \circ \cdots \circ Df_{n+m}(0)^{-1}K_{n+m}^s$, and similarly for $E_n^u$ but with $m\leq 0$.  In the event that we only consider a one-sided infinite sequence of maps, one of the subspaces can be chosen arbitrarily in its cone.
\end{remark}

\begin{remark}
Condition~\ref{C4} is automatic if the sequence of maps is obtained from a diffeomorphism on a compact manifold via local coordinates along a trajectory.  We stress that $\beta_n$ may become arbitrarily large and $\theta_n$ arbitrarily small; moreover the rate at which they become large and small is not required to be subexponential (compare this with the requirements in non-uniform hyperbolicity that sequences of constants be tempered).
\end{remark}

\begin{remark}
If the sequence $f_n$ is obtained from a diffeomorphism $f$ via local coordinates along a trajectory, and if the splitting in Condition \ref{C2} comes from a dominated splitting for $f$, then $\lambda_n^s < \lambda_n^u$ for all $n$.  In this case two nearby choices of $E_n^u$ will have the same asymptotic behaviour as $n\to+\infty$, while there is only one choice of $E_n^s$ for which $\ulim_{n\to\infty} \theta_n > 0$.  Similarly, two nearby choices of $E_n^s$ will have the same asymptotic behaviour as $n\to-\infty$, while there is only one choice of $E_n^u$ for which $\ulim_{n\to-\infty} \theta_n > 0$.

This behaviour in the tangent space still occurs if the splitting is only asymptotically dominated -- that is, if $\sum_{n=1}^N (\lambda_n^u - \lambda_n^s)$ becomes arbitrarily large with $N$, even though individual terms may be negative.  An important part of any Hadamard--Perron theorem is to establish this asymptotic behaviour not just for subspaces in the tangent space, but for submanifolds in $V_n$ itself.
\end{remark}

\begin{remark}\label{rmk:differentn}
The range of values that $n$ takes will vary. 
\begin{enumerate}
\item In Section~\ref{sec:HP2}, we will consider all $n\geq 0$, since Theorem~\ref{thm:HP2} concerns asymptotic behaviour of admissible manifolds as $n\to\infty$.
\item In Section~\ref{sec:C1+unstable}, we will consider all $n\leq 0$, since Theorem~\ref{thm:HP5} concerns true unstable manifolds, which are defined in terms of their asymptotic behaviour under the maps $f_n^{-1}$.
\item In Sections~\ref{sec:parameters}--\ref{sec:finite}, we will consider finitely many $n$, say $0 \leq n \leq N$, since Theorems~\ref{thm:parameters}--\ref{thm:tau} concern images of admissible manifolds under finite compositions of the maps $f_n$.
\end{enumerate}
\end{remark}

We also make the standing assumption that the domain $\Omega_n$ is large enough.  More precisely, once parameters $\tau_n,r_n,\gamma_n$ are specified (see~\eqref{eqn:C}), we have
\begin{enumerate}[leftmargin=3em, label=\textbf{(C\arabic{*})}]\setcounter{enumi}{4}
\item\label{C5}
 $\Omega_n \supset B_n^u(r_n) \times B_n^s(\tau_n+\gamma_n r_n)$,
\end{enumerate}
where $B_n^u(r_n)$ is the ball of radius $r_n$ in $E_n^u$ centred at $0$, and similarly for $B_n^s$.  It will suffice to have $\Omega_n \supset B(0,\eta)$ for some fixed $\eta>0$.

Given $m<n$, we will write
\begin{equation}\label{eqn:Fmn}
F_{m,n} = f_{n-1} \circ f_{n-2} \circ \cdots \circ f_m
\end{equation}
wherever the composition is defined, and we will let $\Omega_m^n$ be the connected component of $\bigcap_{k=m}^{n-1} (F_{m,k})^{-1}(\Omega_k)$ that contains $0$.  We will be concerned exclusively with the action of
\[
F_{m,n} \colon \Omega_m^n \to V_n;
\]
in particular, given any $W\subset V_m$, we will write
\[
F_{m,n}(W) := F_{m,n}(W|_{\Omega_m^n}).
\]


From now on we will use coordinates on $V_n$ given by $E_n^u \oplus E_n^s$: for $x\in V_n$, we write $x = x_u + x_s = (x_u,x_s)$, where $x_u \in E_n^u$ and $x_s \in E_n^s$.  We will usually use the letter $x$ for a point in $V_n$ and the letter $v$ for a vector in $E_n^u$.  We will work with \emph{admissible manifolds} given as graphs of functions $\psi\colon B_n^u(r_n)\subset E_n^u\to E_n^s$, where 
$\graph\psi=\{(v,\psi(v))\mid v\in E_n^u\}$.  

Given sequences of numbers $r_n>0$ (presumed small), $\tau_n, \sigma_n\geq 0$ (also small),  and $\kappa_n>0$ (presumed large), we will be interested in admissible manifolds that arise as graphs of functions in the following class:
\begin{equation}\label{eqn:C}
\begin{aligned}
\CCC_n &= \CCC_n(r_n, \tau_n, \sigma_n, \kappa_n) \\
&= \Big\{ \psi \colon B_n^u(r_n) \to E_n^s \mid \psi \text{ is $C^{1+\alpha}$,}\ \|\psi(0)\| \leq \tau_n,\\
&\qquad\qquad\qquad\qquad\qquad  \|D\psi(0)\|\leq \sigma_n, \text{ and } |D\psi|_\alpha \leq \kappa_n \Big\},
\end{aligned}
\end{equation}
where
\begin{equation}\label{eqn:norm}
|D\psi|_\alpha := \sup_{v_1\neq v_2 \in B_n^u(r_n)} \frac{\|D\psi(v_1) - D\psi(v_2)\|}{\|v_1-v_2\|^\alpha}.
\end{equation}
We will refer to $r_n,\tau_n,\sigma_n,\kappa_n$ collectively as the \emph{parameters} of $\CCC_n$, and will say that they are uniformly bounded on a set $\Gamma\subset \ZZ$ if\footnote{In practice $\tau_n,\sigma_n$ will actually be quite small, and so the battle will be to control $r_n$ and $\kappa_n$.}
\[
\inf_{n\in \Gamma} r_n > 0, \qquad \sup_{n\in \Gamma} \max\{\tau_n,\sigma_n,\kappa_n\}<\infty.
\]

\begin{remark}
If we write $\gamma_n = \sigma_n + \kappa_n r_n^\alpha$, then the conditions in \eqref{eqn:C} imply the bound $\|D\psi\|\leq \gamma_n$ for all $\psi\in \CCC_n$, where
\[
\|D\psi\| := \sup_{v\in B_n^u(r_n)} \|D\psi(v)\|.
\]
In the proofs, and in particular in Theorem \ref{thm:HP1}, we will give results that allow us to consider the space of functions $\psi\in \CCC_n$ that satisfy $\|D\psi\|\leq \gamma_n$ for some (potentially) smaller value of $\gamma_n$.  Our main results (Theorems \ref{thm:HP2}--\ref{thm:tau}) will include the assumption that there is some small $\bar\gamma>0$ such that $\sigma_n + \kappa_n r_n^\alpha\leq \bar\gamma$ for every $n$, so that in particular $\|D\psi\|\leq \bar\gamma$ for all $\psi\in \CCC_n$.
\end{remark}

Let $\WWW_n$ be the space of admissible manifolds corresponding to $\CCC_n$ -- that is,
the collection of submanifolds of $V_n$ that arise as graphs of functions in $\CCC_n$.  If $W=\graph\psi\in \WWW_n$ is such that some relatively open set $U\subset f_n(W)$ is in $\WWW_{n+1}$, then we let $\bar\psi$ be the unique member of $\CCC_{n+1}$ such that $U = \graph\bar\psi$.  We write $G_n \colon \psi \mapsto \bar\psi$ for the corresponding map, called the \emph{graph transform}.

Note that $G_n$ is not necessarily defined on all of $\CCC_n$, since for a given $W\in \WWW_n$, the image $f_n(W)$ need not have any subsets in $\WWW_{n+1}$.  Thus an important part of what follows is to give conditions on the parameters such that $G_n\colon \CCC_n \to \CCC_{n+1}$ is defined on all of $\CCC_n$.  If this is the case for every $n$, then we write
\[
\GGG_n = G_{n-1} \circ G_{n-2} \circ \cdots \circ G_0 \colon \CCC_0 \to \CCC_n.
\]

\subsection{Relations to known results}


In the uniformly hyperbolic setting, the relevant version of the Hadamard--Perron Theorem may be found in~\cite[Theorem 6.2.8]{KH95}; we state a related result as Theorem~\ref{thm:uniform}.  For this version, one makes the following assumptions.
\begin{enumerate}[label=(\roman{*})]
\item\label{unif-exp} Uniform expansion:  $\inf_n \lambda_n^u > 0$.
\item\label{dom-spl} Dominated splitting:  $\inf_n \lambda_n^u > \sup_n \lambda_n^s$.
\item\label{unif-trans} Uniform transversality:  $\inf_n \theta_n > 0$.
\item\label{non-lin-small} $f_n$ is $C^1$ and $\|Df_n(x) - Df_n(0)\|$ is sufficiently small.
\end{enumerate}
Under these assumptions, the local manifolds $W_n^u$ are shown to have uniformly large size.

In the non-uniformly hyperbolic setting, the typical approach is to use Lyapunov coordinates so that~\ref{unif-exp}--\ref{unif-trans} still hold, while the non-linear part $\|Df_n(x) - Df_n(0)\|$ may be large, and in particular~\ref{non-lin-small} is replaced with
\begin{enumerate}[label=(\roman{*}$^\prime$), start=4]
\item\label{non-lin-tempered} $f_n$ is $C^{1+\alpha}$ and $\ulim_{n\to\pm\infty} \frac 1{|n|} \log|Df_n|_\alpha < \alpha \inf_n \lambda_n^u$.
\end{enumerate}
Then one uses the version of the theorem found in~\cite[Theorem 7.5.1]{BP07}, stated below as Theorem~\ref{thm:nonuniform}.  A key difference in the conclusion here is that the size of the $W_n^u$ may decay as $n\to\pm\infty$, although the rate of decay is slower than the rate of contraction or expansion in the dynamics.

When the trajectories to which the non-uniform Hadamard--Perron theorem is applied are generic trajectories for a hyperbolic invariant measure, one can conclude that although the size of the manifolds $W_n$ may become arbitrarily small, it nevertheless recurs to large scale and is bounded away from $0$ on a set of times with positive asymptotic frequency.  However, if one wishes to use some version of the Hadamard--Perron theorem to construct manifolds $W_n$ that can be used in establishing the existence of invariant measures with certain properties, as in~\cite{CDP11b}, then the recurrence to large scale must be established without recourse to ergodic theory.

This idea -- that one may wish to obtain results on admissible manifolds and unstable manifolds without needing to invoke the presence of a specific invariant measure -- is a principal motivator for the results in this paper.  We impose various conditions on the maps $f_n$ under which our results hold: certain conditions hold whenever $f_n$ is a typical sequence of germs for some invariant measure, but we do not require any knowledge about such a measure for the theorems themselves.

We accomplish recurrence to large scale for admissible manifolds in Theorem~\ref{thm:HP2}, where we consider $C^{1+\alpha}$ maps for which \ref{unif-exp}--\ref{unif-trans} may fail.  
We introduce the notion of \emph{effective hyperbolicity} for the sequence $\{f_n\}$; roughly speaking, this requires that the expansion in the unstable direction overcomes the defect from domination and the decay of the angle.  For an effectively hyperbolic sequence of maps, there is a certain sequence of \emph{effective hyperbolic times} along which a sequence of admissible manifolds is well-behaved, and in particular the graph transform
\[
\GGG_n\colon \CCC_0(\bar{r},0,0,\bar\kappa) \to \CCC_n(\bar{r},0,0,\bar\kappa)
\]
is well defined.  These effective hyperbolic times are obtained via Pliss' lemma and are analogous to the well-established notion of hyperbolic times.  However, there is a key difference between these two notions: while at hyperbolic times the derivative of the map acts uniformly hyperbolically on the tangent space, at effective hyperbolic times it is the map itself whose action is locally uniformly hyperbolic.  Although the set of effective hyperbolic times is a subset of the set of hyperbolic times, it nevertheless has positive asymptotic density under the hypotheses of the theorem.

Theorem~\ref{thm:HP5} deals with the unstable manifolds themselves (rather than the admissibles), which exist as soon as the sequence is effectively hyperbolic and are unique as soon as the splitting is asymptotically dominated.  

Theorem~\ref{thm:parameters} gives precise conditions on the parameters $r_n,\tau_n,\sigma_n,\gamma_n,\kappa_n$ for the graph transform to be well-defined, and 
Theorem~\ref{thm:tau} uses effective hyperbolicity to explicitly determine sequences of parameters satisfying the conditions of Theorem~\ref{thm:parameters}.  




\section{Main Results}\label{sec:C1+}

\subsection{Effective hyperbolic times and recurrence to large scale}\label{sec:HP2}


We now describe a setting in which the $\CCC_n$ can be chosen so that the graph transforms are defined for all $n$ and the parameters are uniformly bounded on a set of times with positive asymptotic density.

Our approach is modeled on the notion of \emph{hyperbolic times}, which were introduced by Alves, Bonatti, and Viana in \cite{ABV00}.  These are times $n$ such that the composition $f_{n-1} \circ \cdots \circ f_{k+1} \circ f_k$ has uniform expansion along $E_k^u$ for every $0\leq k< n$.  In our setting, where the splitting $V_n = E_n^u \oplus E_n^s$ may not be uniformly dominated, we must strengthen this notion to that of an \emph{effective hyperbolic time}, where the good properties of the derivative cocycle can be brought back to the maps $f_n$ themselves.  The set of effective hyperbolic times is contained in the set of hyperbolic times, but there may be hyperbolic times that are not effective.

Abundance of hyperbolic times is assured by assuming that $\lambda_n^u$ has asymptotically positive averages.  For abundance of effective hyperbolic times, we introduce a quantity  that depends not just on $\lambda_n^u$, but also on 
$\lambda_n^s$ and $\beta_n$.\footnote{Recall that these are defined in \ref{C3}.}  If this quantity has asymptotically positive averages, then there is a positive frequency of effective hyperbolic times.

Let $\{f_n \mid n\geq 0\}$ satisfy~\ref{C1}--\ref{C5}.  The following quantity may be thought of as the \emph{defect from domination} (recall that $\alpha\in(0,1]$ is the H\"older exponent of $Df_n$):
\begin{equation}\label{eqn:defect}
\Delta_n := \max\left( 0, \frac{\lambda_n^s - \lambda_n^u}{\alpha} \right).
\end{equation}
Note that $\Delta_n=0$ if $\lambda_n^s \leq \lambda_n^u$, which is the case when the splitting $E_n^u \oplus E_n^s$ is dominated.  Fix a threshold value $\bar\beta$ and define 
\begin{equation}\label{eqn:ln}
\lambda_n^e = \begin{cases}
\lambda_n^u - \Delta_n
& \text{ if }\beta_n \leq \bar\beta, \\
\min \left( \lambda_n^u - \Delta_n, \,
\frac 1\alpha \log \frac {\beta_{n-1}}{\beta_n} \right)
& \text{ if }\beta_n > \bar\beta.
\end{cases}
\end{equation}
Obviously $\lambda_n^e$ depends on the choice of $\bar\beta$, but we will suppress this dependence in the notation to minimise clutter.

\begin{definition}\label{def:eh}
The sequence $\{f_n \mid n\geq 0\}$ is \emph{effectively hyperbolic with respect to the splitting $E_n^u \oplus E_n^s$}  if there exists $\bar\beta$ such that
\begin{equation}\label{eqn:lambdabig}
\chi^e := \llim_{n\to\infty} \frac 1n \sum_{k=0}^{n-1} \lambda_k^e > 0.
\end{equation}
\end{definition}

\begin{remark}
See Section~\ref{sec:verifying} for a discussion of ways that effective hyperbolicity can be verified.
\end{remark}

\begin{remark}
It is natural to consider effective hyperbolicity when $E_n^u$ is the full unstable subspace, but the notion can also be applied when $E_n^u$ is a strong unstable subspace corresponding to the largest Lyapunov exponents, or even when $E_n^u$ is a weak unstable subspace and the largest Lyapunov exponents are included in $E_n^s$, provided the expansion in $E_n^u$ overcomes the failure of domination.
\end{remark}

\begin{definition}
Given fixed thresholds $\bar\beta$ and $\hat\chi>0$, we say that $n$ is an \emph{effective hyperbolic time} if
\begin{equation}\label{eqn:eft}
\frac 1{n-k} \sum_{j=k}^{n-1} \lambda_j^e \geq \hat\chi
\end{equation}
for every $0\leq k < n$.
\end{definition}

\begin{remark}
If we replace $\lambda_j^e$ in~\eqref{eqn:eft} with $\lambda_j^u$, then we arrive at the usual definition of hyperbolic time.   Because $\lambda_j^e \leq  \lambda_j^u$, we see that the set of effective hyperbolic times is a (generally proper) subset of the set of hyperbolic times.
\end{remark}

Given a subset $\Gamma\subset \NN$, write $\Gamma_N = \Gamma \cap [0,N)$ and denote the lower asymptotic density of $\Gamma$ by
\[
\ld(\Gamma) = \llim_{N\to\infty} \frac 1N \#\Gamma_N.
\]
The upper asymptotic density $\ud(\Gamma)$ is defined similarly.

\begin{definition}
The splitting $E_n^u \oplus E_n^s$ is \emph{asymptotically dominated} if
\begin{equation}\label{eqn:asymp-dom}
\chi^g := \llim_{n\to\infty} \frac 1n \sum_{k=0}^{n-1} (\lambda_k^u - \lambda_k^s) > 0.
\end{equation}\end{definition}

In this section and the next we will consider the following collection of admissible manifolds for parameters $r,\kappa>0$:
\begin{multline*}
\hat\CCC_n(r,\kappa) = \CCC_n(r,0,0,\kappa) = \{ \psi \colon B_n^u(r) \to E_n^s \mid \psi \in C^{1+\alpha},  \\ \psi(0) = 0,\ D\psi(0)=0,\ 
\ |D\psi|_\alpha \leq \kappa \}.
\end{multline*}

\begin{remark}
As in the definition of $\CCC_n$, note that every $\psi\in\hat\CCC_n$ satisfies $\|D\psi\|\leq \gamma := \kappa r^\alpha$.
\end{remark}

The following theorem shows that the pushforwards of admissible manifolds are well-behaved at the set $\Gamma$ of effective hyperbolic times, and that $\Gamma$ has positive  lower asymptotic density as long as the  asymptotic average rate of effective hyperbolicity is positive.

\begin{thma}\label{thm:HP2}
Given $\bar\beta,L>0$, $\alpha\in(0,1]$, $\hat\chi^u > \bar\chi^u > 0$, and $\hat\chi^g > \bar\chi^g > 0$, the following is true for every sufficiently small $\bar\gamma,\bar{r},\bar\theta>0$ and every sufficiently large $\bar\kappa$ satisfying $\bar\kappa\bar{r}^\alpha\leq\bar\gamma$.    If $\{f_n\mid n\geq 0\}$ satisfies~\ref{C1}--\ref{C5} and is effectively hyperbolic with respect to the splitting $E_n^u\oplus E_n^s$, with $\chi^e > \hat\chi^u$ (using the threshold $\bar\beta$), then
\begin{equation}\label{eqn:hyptimes}
\ld(\Gamma)\geq \frac{\chi^e - \hat\chi^u}{L-\hat\chi^u}>0,
\end{equation}
where $\Gamma$ is the associated set of effective hyperbolic times.  Moreover, the following are true for every $n\in\Gamma$.
\begin{thmenum}
\item $\theta_n \geq \ba$, where $\theta_n$ controls $\measuredangle(E_n^u,E_n^s)$ as in~\eqref{eqn:C3c}.
\item The graph transform $\GGG_n \colon \hat\CCC_0(\bar{r},\bar\kappa) \to \hat\CCC_n(\bar r, \bar\kappa)$ is well-defined; in particular, given $\psi_0 \in \hat\CCC_0$, the $C^{1+\alpha}$ function $\psi_n = \GGG_n \psi_0 \colon B_n^u(\bar{r}) \to E_n^s$ satisfies
\begin{enumerate}
\item $\psi_n(0)=0$, $D\psi_n(0) = 0$, $\|D\psi_n\| \leq \bar\gamma$, and $|D\psi_n|_\alpha \leq \bar\kappa$;
\item $\graph\psi_n = F_{0,n}(\graph \psi_0)$.
\end{enumerate}
\item\label{A.IV} If $x,y\in (\graph\psi_m) \cap \Omega_m^n$ for some $0\leq m\leq n$, then
\begin{equation}\label{eqn:hyptimeHP2}
\|F_{m,n}(x) - F_{m,n}(y)\| \geq e^{(n-m)\bar\chi^u} \|x - y\|.
\end{equation}
\item 
If the splitting $E_n^u \oplus E_n^s$ is asymptotically dominated with $\chi^g > \hat\chi^g$, then for every $\ph_0,\psi_0\in\hat\CCC_0$ we have
\begin{equation}\label{eqn:contraction5}
\ulim_{\stackrel{n\to\infty}{n\in \Gamma}} \frac 1n \log \|\psi_n - \ph_n\|_{C^0} < -\bar\chi^g.
\end{equation}
\end{thmenum}
\end{thma}

The rest of the theorems in this paper give results that apply to times $n\notin \Gamma$ as well.  Roughly speaking, to each $n$ we will associate a constant $M_n\geq 0$ that controls how ``bad'' the dynamics and geometry of the admissible manifolds at time $n$ can be, and which has the property that $M_n=0$ for all $n\in\Gamma$.

\begin{remark}
The formulation of the dependence between the various parameters and constants appearing in Theorem~\ref{thm:HP2} will be echoed throughout the paper.  The meaning of ``sufficiently small'' and ``sufficiently large'' here is that once $\bar\beta,L,\alpha,\hat\chi^{u,g},\bar\chi^{u,g}$ are fixed, there exist $\tilde\gamma,\tilde{r},\tilde\theta,\tilde\kappa>0$ such that if $\bar\gamma\in (0,\tilde\gamma]$, $\bar{r}\in (0,\tilde{r}]$, $\bar\theta\in(0,\tilde\theta]$, and $\bar\kappa\geq\tilde\kappa$, and if in addition $\bar\kappa\bar{r}^\alpha\leq\bar\gamma$, then the rest of the statement of the theorem is valid.  The key point is that $\tilde\gamma,\tilde{r},\tilde\theta,\tilde\kappa$ do not depend on $f_n$ directly, or even on $\lambda_n^{u,s},\lambda_n^e,\beta_n$, but only on $\bar\beta,L,\alpha,\hat\chi^{u,g},\bar\chi^{u,g}$.  One should imagine that $\bar\beta,L$ are very large, since the battle is to control what happens when the non-linearities in $f_n$ become strong.
\end{remark}

\subsection{Effective hyperbolicity and unstable manifolds}\label{sec:C1+unstable}



We consider now a sequence of maps $\{f_n \mid n\leq 0\}$, and using the same notation as in the previous section, make the following definitions that are exact analogues of the definitions there.

\begin{definition}
The sequence $\{f_n \mid n\leq 0\}$ is \emph{effectively hyperbolic with respect to the splitting $E_n^u\oplus E_n^s$} if there exists $\bar\beta$ such that
\begin{equation}\label{eqn:lambdabig2}
\chi^e := \llim_{n\to-\infty} \frac 1{|n|} \sum_{k=n}^{-1} \lambda_k^e > 0.
\end{equation}
\end{definition}

\begin{definition}
The splitting $E_n^u \oplus E_n^s$ for $\{f_n \mid n\leq 0\}$ is \emph{asymptotically dominated} if
\begin{equation}\label{eqn:asymp-domb}
\chigap := \llim_{n\to-\infty} \frac 1{|n|} \sum_{k=n}^{-1} ( \lambda_k^u - \lambda_k^s) > 0,
\end{equation}
\end{definition}

The following quantity will be used to control the size and regularity of the local unstable manifolds; it is finite whenever $\{f_n\}$ is effectively hyperbolic and $\hat\chi^u \in (0,\chi^e)$:
\begin{equation}\label{eqn:Mn}
M_n(\hat\chi^u) := \sup_{m\leq n} \sum_{k=m}^{n-1} (\hat\chi^u - \lambda_k^e).
\end{equation}
As usual, $M_n(\hat\chi^u)$ depends on the choice of threshold $\bar\beta$, but we will suppress this dependence in the notation. 

\begin{thma}\label{thm:HP5}
Given $\bar\beta,L>0$, $\alpha\in(0,1]$, $\hat\chi^u > \bar\chi^u > 0$, and $\hat\chi^g > 0$, the following is true for every sufficiently small $\bar\gamma,\bar{r},\bar\theta>0$ and every sufficiently large $\bar\kappa$ satisfying $\bar\kappa\bar{r}^\alpha\leq\bar\gamma$.    If $\{f_n\mid n\leq 0\}$ satisfies~\ref{C1}--\ref{C5} and is effectively hyperbolic with respect to the splitting $E_n^u\oplus E_n^s$, with $\chi^e > \hat\chi^u$ (using the threshold $\bar\beta$), and if in addition $\beta_m \leq \bar\beta$ for infinitely many $m$, then we have the following conclusions.
\begin{thmenum}
\item The set $\{n\leq 0 \mid M_n(\hat\chi^u) = 0\}$ has lower asymptotic density at least $(\frac{\chi^e - \hat\chi^u}{L-\hat\chi^u})^2>0$.
\item\label{thetangeq} $\theta_n \geq \bar\theta e^{-\alpha M_n(\hat\chi^u)}$ for every $n\leq 0$.
\item\label{inv-mfd}
There exists $\psi_n \in \hat\CCC_n(\bar r e^{-M_n(\hat\chi^u)},\bar\kappa e^{\alpha M_n(\hat\chi^u)})$ such that $f_n(\graph\psi_n)\supset\graph\psi_{n+1}$ for every $n<0$.  In particular, $\psi_n(0)=0$, $D\psi_n(0)=0$, $\|D\psi_n\|\leq \bar\gamma$, $|D\psi_n|_\alpha \leq \bar\kappa e^{\alpha M_n(\hat\chi^u)}$.
\item\label{B-exp} If $x,y\in (\graph\psi_m)\cap \Omega_m^n$ for some $n>m$, then
\begin{equation}\label{eqn:dFn3}
\|F_{m,n}(x) - F_{m,n}(y) \| \geq e^{-M_n(\hat\chi^u)} e^{(n-m)\bar\chi^u} \|x-y\|.
\end{equation}
\item\label{B-unique}  If the splitting $E_n^u\oplus E_n^s$ is asymptotically dominated with $\chigap > \hat\chi^g$, then $\psi_n$ is the unique function in $\hat\CCC_n(\bar r e^{-M_n(\hat\chi^u)},\bar\kappa e^{\alpha M_n(\hat\chi^u)})$ satisfying~\ref{inv-mfd}
\item\label{B-char} 
If in addition to asymptotic domination we have the stronger condition
\begin{equation}\label{eqn:asymp-dom2}
\chi^s := \ulim_{n\to-\infty} \frac 1{|n|} \sum_{k=n}^{-1} \lambda_k^s < \bar\chi^u,
\end{equation}
then $\psi_n$ admits the following characterisation: if $x\in \Omega_n$ and $C\in \RR$ are  such that 
\begin{equation}\label{eqn:backcontr3}
\|F_{m,n}^{-1} (x)\| \leq C e^{-(n-m)\bar\chi^u}
\end{equation}
for every $m$, then $x\in \graph\psi_n$.
\end{thmenum}
\end{thma}

\begin{remark}
Theorem~\ref{thm:HP5} shows that the unstable manifolds have uniformly bounded size, curvature, and dynamical properties on the set of times $\Gamma_M := \{n \mid M_n(\hat\chi^u) \leq M\}$ for each $M\geq 0$.  As $M$ increases, the bounds get worse: size decreases, while curvature and the constant $C$ in~\eqref{eqn:backcontr3} increase.  The trade-off is that it is sometimes possible to guarantee that $\ld(\Gamma_M)$ goes to 1 as $M\to\infty$, in which case we obtain uniform control on a set of times with arbitrarily large lower asymptotic density.
\end{remark}

\subsection{Verifying effective hyperbolicity}\label{sec:verifying}

The quantity $\lambda_n^e$  that appears in the definition of effective hyperbolicity depends on $\lambda_n^u, \lambda_n^s$, and $\beta_n$.  If one has some information about the frequency with which $\beta_n$ becomes large (that is, $|Df_n|_\alpha$ becomes large and/or $\theta_n$ becomes small), then effective hyperbolicity can be verified by considering only $\lambda_n^u$ and $\lambda_n^s$.

To this end, suppose that
\begin{equation}\label{eqn:largebeta}
\lim_{\bar\beta\to \infty} \ud \{ n \mid \beta_n > \bar\beta \} = 0,
\end{equation}
where $\ud$ is upper asymptotic density.
Let $\lambda_n^u, \lambda_n^s$ be as before, and let $\Delta_n$ be the defect from domination defined in~\eqref{eqn:defect}.  Then effective hyperbolicity of $\{f_n\}$ reduces to the condition that
\begin{equation}\label{eqn:eh2}
\chi^u := \llim_{n\to\infty} \frac 1n \sum_{k=0}^{n-1} (\lambda_k^u - \Delta_k) > 0.
\end{equation}
Note that~\eqref{eqn:eh2} does not depend on $\bar\beta$.  We have the following result.

\begin{proposition}\label{prop:verifying}
If a sequence $\{f_n \mid n\geq 0 \}$ satisfies~\eqref{eqn:largebeta} and~\eqref{eqn:eh2}, then it is effectively hyperbolic.  In particular, Theorem~\ref{thm:HP2} applies.
\end{proposition}

Similar observations hold regarding Theorem~\ref{thm:HP5}.  For a sequence $\{f_n \mid n\leq 0\}$, we can replace~\eqref{eqn:eh2} with
\begin{equation}\label{eqn:eh3}
\chi^u := \llim_{n\to-\infty} \frac 1{|n|} \sum_{k=n}^{-1} (\lambda_k^u - \Delta_k) > 0,
\end{equation}
and obtain the following.

\begin{proposition}\label{prop:verifying2}
If a sequence $\{f_n \mid n\leq 0\}$ satisfies~\eqref{eqn:largebeta} and \eqref{eqn:eh3}, then it is effectively hyperbolic and has $\llim_m \beta_m < \infty$.  In particular, Theorem~\ref{thm:HP5} applies.
\end{proposition}

Proposition~\ref{prop:verifying2} allows us to verify effective hyperbolicity by bounding the asymptotic average of $\lambda_n^e$.  However, a computation of the constants $M_n(\hat\chi^u)$ that appear in Theorem~\ref{thm:HP5} (see~\eqref{eqn:Mn}) requires knowledge of $\lambda_n^e$ itself, and not just its asymptotic average.  A slight simplification can be achieved by observing that~\ref{C4} implies the bound $\lambda_n^e \geq -(1+\frac 1\alpha)L =: -L'$, which allows us to forgo computing the exact sum in~\eqref{eqn:Mn} and instead use
\begin{equation}\label{eqn:Mnleq}
M_n(\hat\chi^u) \leq \sup_{m\leq n} \bigg(
(n-m)\hat\chi - \sum_{k=m}^{n-1} \big(\lambda_k^u - \Delta_k - 2L'\one_k(\bar\beta)\big)\bigg),
\end{equation}
where $\one_n(\bar\beta) = 1$ if $\beta_n>\bar\beta$ and is $0$ otherwise.  This has the advantage that the quantities $|Df_n|_\alpha$ and $\theta_n$ enter only through the number of times that the threshold $\bar\beta$ is exceeded, and the rest of the expression depends only on the linear terms $\lambda_n^{u,s}$.  We will use this approach in Section~\ref{sec:closing} to state a closing lemma using effective hyperbolicity.

\subsection{Parameter conditions for a well-defined graph transform}\label{sec:parameters}

Theorems~\ref{thm:HP2} and~\ref{thm:HP5} are both ultimately derived from the following result, which gives more precise conditions on the parameters $r_n,\tau_n,\sigma_n,\kappa_n$ for the graph transform $G_n \colon \CCC_n \to \CCC_{n+1}$ to be well-defined.  Note that now we allow $\tau_n$ and $\sigma_n$ to take positive values, which puts us in a more general setting than the previous sections.

Given $\delta>0$, consider the following recursive relations on the parameters:
\begin{align}
\label{eqn:rec-r}
 r_{n+1} &\leq e^{(\lambda_n^u -\delta)}r_n, \\
 \label{eqn:rec-t}
 \tau_{n+1} &\geq e^{(\lambda_n^s +  \delta)}\tau_n, \\
 \label{eqn:rec-s}
 \sigma_{n+1} &\geq e^{(\lambda_n^s - \lambda_n^u + \delta)}\sigma_n, \\
 \label{eqn:rec-k}
 \kappa_{n+1} &\geq e^{(\lambda_n^s - (1+\alpha)\lambda_n^u + \delta)}\kappa_n.
\end{align}

\begin{remark}
Removing $\delta$ from~\eqref{eqn:rec-r}--\eqref{eqn:rec-k} gives the exact bounds that the parameters would be required to follow if the maps $f_n$ were linear.
\end{remark}

Given $\xi,\bar\gamma>0$, consider the following additional set of bounds:
\begin{align}
\label{eqn:bd-r}
 \beta_n r_n^\alpha &\leq \xi, \\
 \label{eqn:bd-b}
 \beta_n &\leq \xi \kappa_n, \\
 \label{eqn:bd-t}
 \tau_n &\leq r_n, \\
 \label{eqn:bd-k}
 \kappa_n \tau_n^\alpha &\leq \sigma_n, \\
 \label{eqn:bd-s}
 \sigma_n + \kappa_n r_n^\alpha &\leq\bar\gamma.
\end{align}

\begin{thma}\label{thm:parameters}
For every $\delta>0$, $L>0$, and $\alpha\in(0,1]$,  there exist $\xi>0$ and $\bar\gamma>0$ such that the following is true.

For each $0\leq n<N$ let the maps $f_n$ and the parameters $r_n, \kappa_n > 0$, $\tau_n,\sigma_n\geq 0$ be such that~\ref{C1}--\ref{C5} and~\eqref{eqn:rec-r}--\eqref{eqn:bd-s} are satisfied.  Then the following are true.
\begin{thmenum}
\item
The graph transform 
\begin{equation}\label{eqn:graph-transform}
G_n\colon \CCC_n(r_n,\tau_n,\sigma_n,\kappa_n) \to \CCC_{n+1}(r_{n+1},\tau_{n+1},\sigma_{n+1},\kappa_{n+1})
\end{equation}
 is well-defined for each $0\leq n<N$.
\item Given $\psi_0\in \CCC_0$, the $C^{1+\alpha}$ functions $\psi_n=\GGG_n\psi_0\colon B_n^u(r_n)\to E_n^s$ have the following property:  if $x,y\in (\graph\psi_m)\cap\Omega_m^n$ for some $0\leq m\leq n$, then
 \begin{equation}\label{eqn:param-expansion}
  \|F_{m,n}(x) - F_{m,n}(y)\| \geq e^{\sum_{k=m}^{n-1} (\lambda_k^u - \delta)} \|x-y\|.
 \end{equation}
\item Fix $(v_0,w_0)\in \Omega_0^n$ and let $(v_n,w_n) = F_{0,n}(v_0,w_0)$.  Then
 \begin{equation}\label{eqn:nearby-orbit}
 \|w_n - \psi_n(v_n)\| \leq e^{\sum_{k=0}^{n-1} (\lambda_k^s + \delta)} \|w_0 - \psi_0(v_0)\|.
 \end{equation}
 Moreover, if $(v_n',w_n')=F_{0,n}(v_0',w_0')$ is another trajectory such that $\|w_0'-w_0\| \leq\bar\gamma\|v_0'-v_0\|$, then
 \begin{equation}\label{eqn:C-vnvn}
 \|v_n - v_n'\| \geq e^{\sum_{k=0}^{n-1} (\lambda_k^u - \delta)} \|v_0 - v_0'\|.
 \end{equation}
\item \label{C-4}
Given $\psi_0,\ph_0\in\CCC_0(r_0,\tau_0,\sigma_0,\kappa_0)$, the graph transform $\GGG_n\psi_0$ is completely determined by the restriction of $\psi_0$ to $B_0^u(\hat r_n)$, where
\begin{equation}\label{eqn:C-hatr}
\hat{r}_n := e^{\sum_{k=0}^{n-1} (-\lambda_k^u + \delta)} r_n + 3\xi \sum_{k=0}^{n-1} e^{\sum_{j=0}^{k-1} (-\lambda_j^u + \delta)} \tau_k,
\end{equation}
and similarly for $\ph_0$.  In particular, we have
\begin{equation}\label{eqn:param-contraction}
 \|\psi_n - \ph_n\|_{C^0} \leq e^{\sum_{k=0}^{n-1} (\lambda_n^s + \delta)} \left\| (\psi_0 - \ph_0)|_{B_0^u(\hat{r}_n)} \right\|_{C^0}
\end{equation}
\end{thmenum}
\end{thma}

\begin{remark}
Observe that Theorem~\ref{thm:parameters} can be applied to the spaces $\hat\CCC_n$ of admissible manifolds passing through the origin and tangent to $E_n^u$ by taking $\sigma_n=\tau_n=0$.  In this case conditions~\eqref{eqn:rec-r}--\eqref{eqn:bd-s} reduce to
\begin{alignat*}{3}
r_{n+1} &\leq e^{(\lambda_n^u - \delta)} r_n & \beta_n &\leq \xi \min(\kappa_n, r_n^{-\alpha}), \\
\kappa_{n+1} &\geq e^{(\lambda_n^s - (1+\alpha)\lambda_n^u + \delta)}\kappa_n \qquad\qquad & \kappa_n r_n^{\alpha} &\leq \bar\gamma,
\end{alignat*}
and \eqref{eqn:C-hatr} simplifies to $\hat{r}_n := e^{\sum_{k=0}^{n-1} (-\lambda_k^u + \delta)} r_n$.
\end{remark}

\subsection{Finite sequences of diffeomorphisms}\label{sec:finite}



We shall show how the notion of effective hyperbolicity guarantees the existence of sequences of parameters that satisfy both the recursion relations~\eqref{eqn:rec-r}--\eqref{eqn:rec-k} and the bounds~\eqref{eqn:bd-r}--\eqref{eqn:bd-s}, while simultaneously giving good control on the uniformity of $r_n$ and $\kappa_n$.


\begin{thma}\label{thm:tau}
Fix $L,\bar\beta>0$ (presumed large), $\alpha\in(0,1]$, $\chi^u > \hat\chi^u > \bar\chi^u > 0$, and $\hat\chi^s < \bar\chi^s < 0$.  
Then for all sufficiently small $\bar\gamma,\bar{r},\ba>0$, all sufficiently small $\bar{\sigma},\bar\tau\geq 0$, all sufficiently large $\bar\kappa>0$, and all $\hat\kappa\geq\bar\kappa$ such that
\begin{equation}\label{eqn:param-bounds}
\bar\tau\leq\bar{r}, \qquad \hat\kappa\bar\tau\leq\bar\sigma,\qquad \bar\sigma + \hat\kappa\bar{r}^\alpha \leq \bar\gamma,
\end{equation}
every sequence of maps $\{f_n \mid 0\leq n < N\}$ satisfying~\ref{C1}--\ref{C5} and $\beta_0\leq\bar\beta$ has the following properties.
\begin{thmenum}
\item \label{Mus}
For $0\leq n\leq N$, let $M_n^u\geq 0$ be such that
\begin{equation}\label{eqn:M-hyp}
\sum_{k=m}^{n-1} \lambda_k^e \geq (n-m)\hat\chi^u - M_n^u
\end{equation}
for all $0\leq m < n$, and let $M_0^s$ be such that
\begin{equation}\label{eqn:Ms}
\sum_{k=0}^{n-1}  \lambda_k^s \leq n\hat\chi^s - M_n^u + M_0^s
\end{equation}
for all $0\leq n\leq N$.  Then $\theta_n\geq \bar\theta e^{-\alpha M_n^u}$ and the graph transform
\begin{multline}\label{eqn:D-Gn}
\GGG_n\colon \CCC_0(\bar{r},\bar\tau e^{-M_0^s}, \bar\sigma e^{-\alpha M_0^s},\hat\kappa)\\
\to \CCC_n(\bar{r}e^{-M_n^u},\bar\tau e^{-M_n^u} e^{n\bar\chi^s}, \bar\sigma e^{\alpha n \bar\chi^s}, \bar\kappa e^{\alpha M_n^u})
\end{multline}
is well-defined whenever $\hat\kappa \leq \bar\kappa e^{\alpha n \bar\chi^u}$. 
\item \label{Mus2}
Given $\psi_0\in \CCC_0$, the $C^{1+\alpha}$ functions $\psi_n=\GGG_n\psi_0\colon B_n^u(\bar{r}e^{-M_n^u})\to E_n^s$ have the following properties:  if $x,y\in (\graph\psi_m)\cap\Omega_m^n$ for some $0\leq m\leq n$, then
\begin{equation}\label{eqn:hyptimeHP2d}
\|F_{m,n}(x) - F_{m,n}(y)\| \geq e^{-M_n^u} e^{(n-m)\bar\chi^u} \|x-y\|,
\end{equation}
and the same bound applies to the projections to the unstable subspace.
\item \label{D-3}
For $(v_0,w_0)\in \Omega_0^n$ and $(v_n,w_n) = F_{0,n}(v_0,w_0)$, we have
\begin{equation}\label{eqn:nearby-pts2}
\|w_n - \psi_n(v_n)\| \leq e^{M_0^s - M_n^u + n\bar\chi^s} \|w_0 - \psi_0(v_0)\|.
\end{equation}
\item \label{Mus3}
Given $\psi_0,\ph_0\in\CCC_0$, the graph transform $\GGG_n\psi_0$ is completely determined by the restriction of $\psi_0$ to $B_0^u(\hat r_n)$, where $\hat r_n = e^{-n\bar\chi^u} e^{M_n^u} \bar{r} + \bar\tau$, and similarly for $\ph_0$. 
In particular, we have 
\begin{equation}\label{eqn:chius}
\|\psi_n - \ph_n\|_{C^0} \leq 
e^{n\bar\chi^s} e^{M_0^s} (3\bar\tau e^{-M_n^u} + 2 \bar{r}e^{-n\bar\chi^u}),
\end{equation}
\item \label{hyp-times3}
If $\frac 1N \sum_{k=0}^{N-1} \lambda_k^e \geq \chi^u$, then there exists a set $\Gamma\subset [1,N]$ with $\#\Gamma \geq \left(\frac{\chi^u - \hat\chi^u}{L-\hat\chi^u}\right) N$ for which every $n\in \Gamma$ has $\frac 1{n-m}\sum_{k=m}^{n-1} \lambda_k^e \geq \hat\chi^u$ for every $0\leq m<n$, and hence statements \ref{Mus}--\ref{Mus2} apply with $M_n^u=0$.
\end{thmenum}
\end{thma}

\begin{remark}
Note that in \ref{hyp-times3}, we have $M_n^u>0$ for $n\notin \Gamma$, and so in particular $M_n^u$ cannot be omitted in~\eqref{eqn:Ms}, which deals with all $n$, not just $n\in\Gamma$.
\end{remark}

\begin{remark}
The statement of Theorem~\ref{thm:tau} simplifies somewhat if one sets $\bar\sigma=\bar\tau=0$ and considers only admissible manifolds passing through $0$ and tangent to $E_n^u$.  In this case no hypotheses on $\lambda_n^s$ are needed (note that in the domain of $\GGG_n$ in \eqref{eqn:D-Gn}, all the terms containing $M_0^s$ vanish), and in particular \eqref{eqn:Ms} can be omitted.  This version of the result suffices to prove Theorem~\ref{thm:HP2} 
and thus is well-suited to proving existence of SRB measures.
\end{remark}


\begin{remark}
When applying Theorem~\ref{thm:tau} to an infinite sequence $f_n$, positivity of the asymptotic average of $\lambda_k^e$ guarantees effective hyperbolicity in the unstable direction, and the constants $M_n^u$ from~\eqref{eqn:M-hyp} control the non-uniformity of this hyperbolicity.  In principle, negativity of the asymptotic average of $\lambda_k^s$ leads to contraction in the stable direction; we see from~\eqref{eqn:Ms} that to realise this contraction, one actually needs $\sum_{k=0}^{n-1} \lambda_k^s$ to grow more quickly than the constants $M_n^u$.
\end{remark}

\section{Effectively hyperbolic diffeomorphisms of compact manifolds}\label{sec:C1+mfd}

Let $\MMM$ be a compact Riemannian manifold, $U\subset \MMM$ an open set, and $f\colon U\to \MMM$ a $C^{1+\alpha}$ diffeomorphism onto its image, where $\alpha\in(0,1]$.  
Shrinking $U$ if necessary, we can assume that $f$ can be extended to a diffeomorphism from a neighbourhood of $\overline{U}$ to its image.  Then there is an $L>0$ such that for every $x\in U$ and $v,w\in T_x\MMM$, we have
\begin{equation}\label{eqn:closingL}
\begin{gathered}
e^{-L} \leq \frac{\|Df(x)(v)\|}{\|v\|} \leq e^L, \\
 e^{-L} \leq \frac{\sin\measuredangle(Df(x)(v),Df(x)(w))}{\sin\measuredangle(v,w)} \leq e^L, \\
 |Df(x)|_\alpha\leq L.
\end{gathered}
\end{equation}

Let $X\subset U$ be a backwards $f$-invariant set (that is, $f^{-1}X\subset X$).  Assume that on $X$, the tangent bundle has a $Df$-invariant splitting $T_x \mathcal{M} = E^u(x) \oplus E^s(x)$.  The set $X$ may be just a single orbit, and the splitting does not need to be continuous. Given $x\in X$, let
\[
\theta(x) = \measuredangle(E^u(x),E^s(x)).
\]
Writing
\[
\lambda^u(x) = \log\|Df(x)|_{E^u(x)}^{-1}\|^{-1}, \qquad\qquad
\lambda^s(x) = \log\|Df(x)|_{E^s(x)}\|,
\]
denote the defect from domination at $x$ by 
$\Delta(x)=\max\left(0,\frac{\lambda^s(x)-\lambda^u(x)}{\alpha}\right)$.  Fix $\ba>0$ and let
\begin{equation}\label{eqn:lex}
\lambda^e(x) = \min\left(\lambda^u(x) - \Delta(x),\, \frac 1\alpha \log \frac{\sin\theta(f(x))}{\sin\theta(x)}\right)
\end{equation}
whenever $\theta(f(x))<\ba$, and $\lambda^e(x) = \lambda^u(x) - \Delta(x)$ otherwise.

\begin{definition}
We call a diffeomorphism $f$ \emph{effectively hyperbolic} on $X$ if there exists $\bar\theta>0$ such that
\begin{equation}\label{eqn:lambdabig3}
\chi^e := \inf_{x\in X} \llim_{m\to\infty} \frac 1m \sum_{k=1}^m \lambda^e(f^{-k}x) > 0.
\end{equation}
In this case for $\hat\chi\in(0,\chi^e)$ we define $M(x)\geq 0$ by
\[
M(x) = \sup_{m\geq 0} \sum_{k=1}^m \left(\hat\chi - \lambda^e(f^{-k}x)\right).
\]
\end{definition}

Finally, let
\begin{equation}\label{eqn:chisleqchiu}
\chi^s := \sup_{x\in X} \ulim_{m\to\infty} \frac 1m \sum_{k=1}^m \lambda^s(f^{-k}x).
\end{equation}

The following result can be viewed as an Unstable Manifold Theorem for effectively hyperbolic diffeomorphisms.

\begin{theorem}\label{thm:HP6}
Given $L>0$ and $0<\bar\chi<\hat\chi$, the following is true for every sufficiently small $\bar{r},\bar\theta,\bar\gamma>0$ and every sufficiently large $\bar\kappa>0$ satisfying $\bar\kappa \bar{r}^\alpha \leq \bar\gamma$.  If $f$ satisfies \eqref{eqn:closingL} and is effectively hyperbolic on $X$ with $\chi^e > \hat\chi$, then we have the following conclusions.
\begin{thmenum}
\item\label{I} For every $x\in X$, the set $\{n\leq 0 \mid M(f^nx)=0\}$ has positive lower asymptotic density.  
\item $\theta(x) \geq \bar\theta e^{-\alpha M(x)}$.
\item\label{III} There exists a family of submanifolds $\{ W^u(x) \mid x\in X\}$ tangent to $E^u(x)$ such that  $f(W^u(x)) \supset W^u(f(x))$, and each $W^u(x)$ is the image under the exponential map $\exp_x$ of the graph of a $C^{1+\alpha}$ function $\psi_x\colon B^u(x)(e^{-M(x)}\bar{r}) \to E^s(x)$ with $\psi_x(0)=0$, $D\psi_x(0)=0$, $\|D\psi_x\|\leq\bar\gamma$, and $|D\psi_x|_\alpha \leq \bar\kappa e^{\alpha M(x)}$.
\item\label{IV} Given $y,z\in W^u(x)$, we have for all $m\geq 0$
\[
d(f^{-m} y, f^{-m} z) \leq e^{M(x)} e^{-m \bar\chi} d(y,z).
\]
\item\label{V} If $\chi^s < \bar\chi$, then $W^u(x)$ is the unique family  satisfying~\ref{III}
\item\label{VI} If $\chi^s < \bar\chi$, and if $x\in X, y\in \mathcal{M}$ are such that there exists $C\in \RR$ with
\[
d(f^{-m}y,f^{-m}x) \leq \min( \bar{r}e^{-M(x)}, Ce^{-m\bar\chi})
\]
for all $m\geq 0$, then $y\in W^u(x)$.
\end{thmenum}
\end{theorem}

\begin{remark}
One can obtain local stable manifolds by applying Theorem \ref{thm:HP6} to $f^{-1}$.  Note that this requires $f^{-1}$ to be effectively hyperbolic on the trajectories in question, which is a separate issue from effective hyperbolicity of $f$.  Note also that $U$ is not required to be a trapping region for either $f$ or $f^{-1}$ -- all that is needed is for the entire forward (backward) trajectory of points in $X$ to remain in $U$.
\end{remark}

As in Section~\ref{sec:verifying}, we describe some conditions that guarantee effective hyperbolicity.

\begin{proposition}\label{prop:verifying3}
If $f\colon \mathcal{M}\to\mathcal{M}$ is a $C^{1+\alpha}$ diffeomorphism satisfying
\begin{equation}\label{eqn:angle-pos}
\lim_{\ba\to 0} \ulim_{m\to\infty} \frac 1m \# \{1\leq k\leq m \mid \theta(f^{-k}x) < \ba \} = 0
\end{equation}
and
\begin{align}\label{eqn:eff-hyp-pos}
\chi^u &:= \inf_{x\in X} \llim_{m\to\infty} \frac 1m \sum_{k=1}^m \left(\lambda^u(f^{-k}x) - \Delta(f^{-k}x)\right) > 0
\end{align}
on a backward invariant set $X$,  then it is effectively hyperbolic on $X$, and for every $0<\bar\chi<\hat\chi<\chi^u$ there exist $\bar\gamma,\bar{r},\bar\theta,\bar\kappa>0$ such that~\ref{I}--\ref{IV} of Theorem~\ref{thm:HP6} hold.  If $\chi^s<\bar\chi$, then \ref{V}--\ref{VI} hold as well.
\end{proposition}

Theorem~\ref{thm:HP6} may be interpreted as giving concrete estimates on the constants that appear in Pesin theory, which vary according to the regular set that a point lies in, and which control the geometric and dynamical properties of the stable and unstable manifolds.  In Section~\ref{sec:nuh} we discuss some of the differences between the non-uniform hyperbolicity appearing in that theory and the effective hyperbolicity we use here.

\section{Application I: Constructing SRB measures for general non-uniformly attractors}\label{sec:srb}

In~\cite{CDP11b}, Theorem~\ref{thm:HP2} is used as a crucial part of the proof of existence of SRB measures under some very general conditions.  We briefly describe this result here, as Theorem \ref{srb-measure} below.  We note that Theorem \ref{srb-measure} establishes the existence of an SRB measure for the systems considered in \cite{ABV00}, as well as for some new examples \cite{CDP11b}.

As in the previous section, let $\MMM$ be a compact manifold, $U\subset \MMM$ an open set, and $f\colon U\to \MMM$ a $C^{1+\alpha}$ diffeomorphism onto its image for some $\alpha\in (0,1]$.  Now we also assume that $U$ is a trapping region -- that is, $\overline{f(U)} \subset U$.  This implies that \eqref{eqn:closingL} is satisfied for some $L>0$ on $f(U)$.

Suppose that there exists a forward-invariant set $X\subset U$ of positive Lebesgue measure with two measurable transverse cone families $K^s(x), K^u(x) \subset T_xM$ such that
\begin{enumerate}
\item $\overline{Df(K^u(x))} \subset K^u(f(x))$ for all $x\in X$;
\item $\overline{Df^{-1}(K^s(f(x)))} \subset K^s(x)$ for all $x\in f(X)$.
\end{enumerate}
As discussed in Remark \ref{rmk:cones}, the cone families $K^{s,u}$ can be used to obtain an invariant splitting $T_x \MMM = E^u(x) \oplus E^s(x)$ on $X$.  In particular, we will be able to apply Theorem \ref{thm:HP2} after verifying some further conditions.

Define $\lambda^u, \lambda^s\colon X\to \RR$ by
\begin{align*}
\lambda^u(x) &= \inf \{\log \|Df(v)\| \mid v\in K^u(x), \|v\|=1 \}, \\
\lambda^s(x) &= \sup \{\log \|Df(v)\| \mid v\in K^s(x), \|v\|=1 \}.
\end{align*}
Denote the angle between the boundaries of $K^s(x)$ and $K^u(x)$ by
\[
\theta(x) = \inf\{\measuredangle(v,w) \mid v\in K^u(x), w\in K^s(x) \},
\]
and let
\[
\ud_K(x) := \lim_{\ba\to 0} \ud \{ n\geq 1 \mid \theta(f^n(x)) < \ba \}.
\]
Let $\Delta(x) = \max \left(0,\frac{\lambda^s(x) - \lambda^u(x)}\alpha\right)$ be the defect from domination, and let
\begin{align*}
\chi^u(x) &:= \llim_{n\to \infty} \frac 1n \sum_{k=0}^{n-1} \left(\lambda^u(f^kx) - \Delta(f^kx)\right), \\
\chi^s(x) &:= \ulim_{n\to \infty} \frac 1n \sum_{k=0}^{n-1} \lambda^s(f^kx).
\end{align*}
Let $S = \{ x\in X \mid \ud_K(x) = 0, \chi^u(x) > 0, \chi^s(x) < 0 \}$, so that points in $S$ have (forward) trajectories on which $f$ is effectively hyperbolic and has negative Lyapunov exponents in the stable direction.

\begin{theorem}[\cite{CDP11b}]\label{srb-measure}
If $\Leb S>0$, then $f$ has a hyperbolic SRB measure supported on $\Lambda = \bigcap_{n\geq 0} f^n(U)$.
\end{theorem}
\begin{proof}[Sketch of proof]
The idea behind the proof of Theorem~\ref{srb-measure} is to construct an invariant measure $\mu$ as a limit point of the sequence of measures
\begin{equation}\label{eqn:mnSRB}
m_n := \frac 1n \sum_{k=0}^{n-1} f_*^k \Leb,
\end{equation}
and then show that some ergodic component of $\mu$ is an SRB measure.  Using Theorem~\ref{thm:HP2}, one can do this by guaranteeing that the measures $m_n$ give uniformly positive weight to a certain compact subset of the class of ``SRB--like'' measures.

More precisely, one fixes parameters $\theta,r,\kappa$ and lets $\WWW$ be the class of submanifolds of $\MMM$ obtained as $\exp_x\graph\psi$ for some $x\in X$ and $\psi\in \hat\CCC_x(r,\kappa)$ with $\measuredangle(E^u(x),E^s(x))\geq \theta$.  Then $\WWW$ is a compact space of geometrically constrained submanifolds.  To constrain the dynamics while retaining compactness, fix $N,C,\lambda$ and let
\begin{multline*}
\AAA = \{W\in \WWW \mid d(f^{-n}y,f^{-n}z) \leq C e^{-\lambda n} d(y,z)\\
 \text{ for all } y,z\in W \text{ and } 0\leq n\leq N\}.
\end{multline*}
Then fixing $K>0$, one considers the set $\mathcal{B}$ of all Borel measures $\mu$ that can be represented as
\[
\int \ph(x)\,d\mu(x) = \int_\AAA \left(\int_W \ph(x) \rho_W(x)\,dm_W(x)\right)\,d\eta(W),
\]
where $\eta$ is a probability measure on $\AAA$, $m_W$ is volume on $W$, and $\rho_W\colon W\to[1/K,K]$ is a $C^\alpha$ function with $|\rho_W|_\alpha\leq K$.

Using general arguments from smooth ergodic theory, it can be shown that if $m$ is an ergodic invariant measure for which there is a non-zero measure $\mu\in\mathcal{B}$ such that $\mu\leq m$, then $m$ is an SRB measure.  Thus, returning to the measures $m_n$ from \eqref{eqn:mnSRB}, if one can find measures $\mu_n \leq m_n$ such that $\mu_n\in \mathcal{B}$ and $\mu_n\not\to 0$, compactness of $\mathcal{B}$ can be used to find $\mu\in\mathcal{B}$ such that $\mu\leq m$ for some ergodic component $m$ of a limit point of $m_n$, showing that $m$ is an SRB measure.

The measures $\mu_n$ can be found using Theorem~\ref{thm:HP2}.  Writing $S_n$ for the set of points in $S$ for which $n$ is an effective hyperbolic time, the bounds on frequency of effective hyperbolic times show that $\Leb S_n \not\to 0$.  Then one can put $\mu_n = \frac 1n \sum_{k=0}^{n-1} f_*^k (\Leb|_{S_k})\leq m_n$ and use the bounds from Theorem~\ref{thm:HP2} on the graph transform at effective hyperbolic times to show that $f_*^k (\Leb|_{S_k})\in \mathcal{B}$ and hence $\mu_n\in \mathcal{B}$.
\end{proof}

\section{Application II: A finite-information closing lemma}\label{sec:closing}

For uniformly hyperbolic systems, the Anosov closing lemma establishes the existence of a periodic orbit close to any almost-periodic orbit.  More precisely, one has the following result \cite[Theorem 6.4.15]{KH95}.\footnote{In fact the result in \cite{KH95} is somewhat stronger and allows $x,f(x),\dots,f^p(x)$ to be an $\eps$-pseudo-orbit.  Moreover, there is a constant $C$ (independent of $\delta$) such that one can take $\eps = \delta/C$.}

\begin{theorem}[Uniform closing lemma]
Let $\Lambda$ be a (uniformly) hyperbolic set for a $C^1$ diffeomorphism $f$.  Then for every $\delta>0$ there is $\eps>0$ such that for any $x\in \Lambda$ and $p\in\NN$ with $d(x,f^p(x))<\eps$, there exists $z\in B(x,\delta)$ such that $z$ is a hyperbolic periodic point for $f$ with period $p$.
\end{theorem}

A similar result holds for non-uniformly hyperbolic systems \cite[\S 3]{aK80}.  A non-uniformly hyperbolic set $\Lambda$ has a filtration $\Lambda = \bigcup_{K>0} \Lambda_K$, where the sets $\Lambda_K$ are compact but non-invariant, and the parameter $K$ may be thought of as controlling the amount of non-uniformity in the trajectory of $x\in \Lambda_K$, with larger values of $K$ corresponding to worse non-uniformities.

\begin{theorem}[Non-uniform closing lemma]\label{thm:nuclosing}
Let $\Lambda$ be a non-uniformly hyperbolic set for a $C^{1+\alpha}$ diffeomorphism $f$.  Then for every $\delta>0$ and $K>0$ there is $\eps>0$ such that for any $x\in \Lambda_K$ and $p\in \NN$ with $f^p(x) \in \Lambda_K \cap B(x,\eps)$, there exists $z\in B(x,\delta)$ such that $z$ is a hyperbolic periodic point for $f$ with period $p$.
\end{theorem}

The difficulty in applying Theorem \ref{thm:nuclosing} is that determining the non-uniformity constant $K$ associated to some point $x$ requires an infinite amount of information, because $K$ depends on the entire forward and backward trajectory of $x$.  Here we use effective hyperbolicity to give a set of criteria for existence of a nearby hyperbolic periodic orbit that can be verified with a finite amount of information, since they depend only on the action of $f$ near the points $x,f(x),\dots,f^p(x)$.

As in the previous sections, let $\MMM$ be a compact Riemannian manifold and $f\colon U\to \MMM$ a $C^{1+\alpha}$ diffeomorphism from an open set $U$ onto its image.  By shrinking $U$ if necessary, we can extend $f$ to a neighbourhood of $\overline{U}$ so that \eqref{eqn:closingL} holds for some uniform $L>0$.

\begin{definition}\label{def:eh-seg}
We say that an orbit segment $\{x,f(x),\dots,f^p(x)\}\subset U$ is \emph{completely effectively hyperbolic with parameters $M^s,M^u,\hat M^s,\hat M^u>0$, rates $\hat\chi^s < 0 < \hat\chi^u$, and threshold $\bar\theta>0$} if there are $Df$-invariant transverse cone families $K^s,K^u$ on $\{x,f(x),\dots,f^p(x)\}$ such that defining $\lambda^u,\lambda^s,\theta$ as in the previous section and writing $\one_{\bar\theta}$ for the indicator function of the set $\{z \mid \theta(z) < \bar\theta\}$, we have
\begin{equation}\label{eqn:theta-ends}
\theta(x) \geq \bar\theta, \qquad \theta(f^p(x))\geq \bar\theta,
\end{equation}
and the quantities
\begin{align}
\label{eqn:ceh-Mnu}
M_n^u &= \max_{0\leq m<n} \left( (n-m)\hat\chi^u - \sum_{k=m}^{n-1} (\lambda^u - \Delta - L\one_{\bar\theta})(f^kx) \right), \\
\label{eqn:ceh-Mns}
M_n^s &= \max_{n<m\leq p} \left( (n-m)\hat\chi^s + \sum_{k=m}^{n-1} (\lambda^s + \Delta + L\one_{\bar\theta})(f^kx) \right)
\end{align}
satisfy
\begin{align}
\label{eqn:ceh-Mu}
M^u &\geq M_p^u, \\
\label{eqn:ceh-Ms}
M^s &\geq M_0^s.
\end{align}
Moreover, we require that
\begin{align}
\label{eqn:ceh-hMu}
\hat M^u &\geq M_n^s - \sum_{k=n+1}^p (\lambda_k^u - \hat\chi^u) \qquad\text{for all }0< n\leq p, \\
\label{eqn:ceh-hMu2}
\hat M^u &\geq M^s - \sum_{k=1}^p (\lambda_k^u - \hat\chi^u),
\end{align}
and
\begin{align}
\label{eqn:ceh-hMs}
\hat M^s &\geq M_n^u + \sum_{k=0}^{n-1} (\lambda_k^s - \hat\chi^s) \qquad\text{for all }0\leq n<p, \\
\label{eqn:ceh-hMs2}
\hat M^s &\geq M^u + \sum_{k=0}^{p-1} (\lambda_k^s - \hat\chi^s).
\end{align}
\end{definition}

\begin{remark}
We stress again that Definition~\ref{def:eh-seg} only requires verifying a finite amount of information: the cones $K^s,K^u$ do not need to be invariant along the entire trajectory of $x$, but only along $p$ iterates of it, and no asymptotic quantities (such as Lyapunov exponents or Lyapunov charts) need to be computed.
\end{remark}

We can use Theorem~\ref{thm:tau} to prove the following closing lemma regarding completely effectively hyperbolic orbit segments.

\begin{theorem}\label{thm:closing}
Given $L,M^u,M^s,\hat M^u,\hat M^s\in\RR$, $\hat\chi^s<0<\hat\chi^u$, and $\bar\theta, \delta>0$, there exist $\eps>0$ and $p_0\in \NN$ such that if $f\colon U \to \mathcal{M}$ satisfies \eqref{eqn:closingL}, then the following is true.  If $x\in U$ and $p\in \NN$ are such that 
\begin{enumerate}
\item $p\geq p_0$ and the orbit segment $\{x,f(x),\dots,f^p(x)\}\subset U$ is completely effectively hyperbolic with parameters $M^s,M^u,\hat M^s,\hat M^u$, rates $\hat\chi^s,\hat\chi^u$, and threshold $\bar\theta$;
\item $d(x,f^px)<\eps$, and there exist maximal-dimensional subspaces $E^u \subset K^u(x)$, $E^s\subset K^s(x)$ such that $d(Df^p(E^\sigma),E^\sigma)<\eps$ for $\sigma=s,u$,
\end{enumerate}
then there exists a hyperbolic periodic point $z=f^pz$ such that $d(x,z)<\delta$.  Moreover, writing $\hat E^s,\hat E^u$ for the stable and unstable subspaces of $Df^p(z)$, we have $d(\hat E^\sigma,E^\sigma)<\delta$ for $\sigma=s,u$.
\end{theorem}

We give a brief sketch of the argument -- a more detailed proof is in Section~\ref{sec:appspfs}.  Let $W^u,W^s$ be $u$- and $s$-admissible manifolds through $x$, respectively.  For an appropriate choice of $r>0$, the hypotheses are enough to guarantee that $f^{np}(W^u) \cap B(x,r)$ converges to a $u$-admissible manifold near $x$ as $n\to\infty$, and similarly, $f^{np}(W^s) \cap B(x,r)$ converges to an $s$-admissible manifold near $x$ as $n\to-\infty$.  The intersection of these limiting manifolds is the desired periodic point.

\section{General results on admissible manifolds}\label{sec:C1}

We begin the proofs by formulating and proving our most  general result, which is Theorem~\ref{thm:HP1}, a very broad version of the Hadamard--Perron theorem that gives detailed bounds on the dynamics of the graph transform operator (central to Hadamard's method).  This result applies even to finite sequences of $C^1$ diffeomorphisms and gives bounds on the images of admissible manifolds.

In Theorem~\ref{thm:HP3}, we use Theorem~\ref{thm:HP1} to prove the existence of local unstable manifolds (not just admissible manifolds) for a sequence of $C^1$ diffeomorphisms $\{f_n \mid n\leq 0\}$.  In particular, this implies the classical Hadamard--Perron theorems (Theorems \ref{thm:uniform} and \ref{thm:nonuniform}), which give existence of local unstable manifolds in the uniformly and non-uniformly hyperbolic settings.  As with the classical results, we also obtain the existence of local strong unstable manifolds corresponding to the directions with the fastest expansion, which are important in various settings including partial hyperbolicity and maps with dominated splittings.  Applying the same result to the inverse maps $f_n^{-1}$ gives the local stable manifolds.

\subsection{Admissible manifolds: control of the graph transform}\label{sec:HP1}

Given $\psi_n\colon E_n^u\to E_n^s$, a continuous non-decreasing function $Z_n^\psi\colon \RR^+ \to \RR^+$ with $Z_n^\psi(0)=0$ is a \emph{modulus of continuity} for $D\psi_n$ if 
\begin{equation}\label{eqn:moduluspsi}
\|D\psi_n(v_1) - D\psi_n(v_2)\| \leq Z_n^\psi(t) \text{ whenever } \|v_1 - v_2\| \leq t.
\end{equation}
Given a sequence of such functions $Z_n^\psi$, we generalise~\eqref{eqn:C} to the following collection of admissible manifolds:
\begin{equation}\label{eqn:C'}
\begin{aligned}
\CCC_n' &= \CCC_n'(r_n,\tau_n,\sigma_n,\gamma_n,Z_n^\psi) \\
&= \Big\{ \psi \colon B_n^u(r_n) \to E_n^s \mid \psi \text{ is $C^1$,}\ \|\psi(0)\| \leq \tau_n,\ \|D\psi(0)\|\leq \sigma_n,\\
&\qquad \|D\psi\| \leq \gamma_n,  \text{ and $Z_n^\psi$ is a modulus of continuity for $D\psi_n$} \Big\}.
\end{aligned}
\end{equation}
Note that setting $Z_n^\psi(t) = \kappa_n t^\alpha$ and taking $\gamma_n \geq \sigma_n + \kappa_n r_n^\alpha$ recovers the earlier definition of $\CCC_n$.



Consider a sequence of $C^1$ maps $\{f_n \mid 0\leq n<N\}$: replace \ref{C1} with
\begin{enumerate}[label=\textbf{(C\arabic{*}$^\prime$)}]
\item\label{C1'} $f_n\colon \Omega_n\to V_{n+1}$ is a $C^1$ diffeomorphism onto its image, and $f_n(0)=0$.
\end{enumerate}
Similarly, replace \ref{C3} with
\begin{enumerate}[label=\textbf{(C\arabic{*}$^\prime$)},start=3]
\item\label{C3'} The numbers $\lambda_n^u,\lambda_n^s,\theta_n$ satisfy~\eqref{eqn:C3a}--\eqref{eqn:C3c}, and $Z_n^f\colon \RR^+\to\RR^+$ is a modulus of continuity for $Df_n$.
\end{enumerate}
For brevity, we say that the maps $\{f_n\mid 0\leq n<N\}$ satisfy \Cp\ whenever they satisfy \ref{C1'}, \ref{C2}, \ref{C3'}, \ref{C4}, and \ref{C5}, and we write
\begin{equation}\label{eqn:hZ}
\hZ_n^f(t) = Z_n^f(t) (\sin \theta_{n+1})^{-1}.
\end{equation}


In order to control the behaviour of the graph transform in terms of $\lambda_n^{u,s}, \theta_n$, we introduce a number of quantities that can be made arbitrarily small by an appropriate choice of $\tau_n,r_n,\sigma_n,\gamma_n$ in the definition of $\CCC_n$.

First note that if $\psi\in\CCC_n'$ and $x\in \graph\psi$, then
\begin{equation}\label{eqn:Un}
\|x\| \leq \tau_n + r_n(1+\gamma_n).
\end{equation}
Suppose $\tau_n$, $\gamma_n, r_n$ are small enough so that
\begin{equation}\label{eqn:rnbn}
\eps_n^f := \hZ_n^f(\tau_n + r_n(1+\gamma_n)) < e^{\lambda_n^u} (1+\gamma_n)^{-1}.
\end{equation}
Define $\chi_n < \hl_n^u < \lambda_n^u$ and $\cl_n^s, \hl_n^s > \lambda_n^s$ by
\begin{alignat}{3}
\label{eqn:hlnu}
e^{\hl_n^u} &= e^{\lambda_n^u} - \eps_n^u, \qquad &
\eps_n^u &= (1+\gamma_n) \eps_n^f, \\
\label{eqn:hlns}
e^{\hl_n^s} &= e^{\lambda_n^s} +  \eps_n^s, &
\eps_n^s &= \max\{ 1+ \gamma_n^{-1}, 1+\gamma_{n+1} \}\cdot \eps_n^f, \\
\label{eqn:chin}
\chi_n &= \hl_n^u + \eps_n^\chi, &
\eps_n^\chi &= \log \max \left( \frac{1-\gamma_{n+1}}{1+\gamma_n},\, \frac{\sin \theta_{n+1}}{1+\gamma_n} \right), \\
\label{eqn:clns}
e^{\cl_n^s} &= e^{\lambda_n^s} + \check\eps_n, &
\check\eps_n &= (1+e^{\lambda_n^s - \hl_n^u}\gamma_n) \eps_n^f + (1+\gamma_n)e^{-\hl_n^u}(\eps_n^f)^2.
\end{alignat}
Let
\begin{equation}\label{eqn:rhon}
\rho_n(t) = e^{-\hl_n^u}(1 + e^{\lambda_n^s - \hl_n^u} \gamma_n) \hZ_n^f(t) + e^{-2\hl_n^u} \hZ_n^f(t)^2
\end{equation}
and suppose that the moduli of continuity $Z_n^\psi$ satisfy
\begin{equation}\label{eqn:kn}
Z_{n+1}^\psi(t e^{\hl_n^u}) \geq
e^{\lambda_n^s - \hl_n^u} Z_n^\psi(t) + \rho_n(t).
\end{equation}
Finally, write
\[
\eps_n^\sigma = e^{\lambda_n^s - \hl_n^u} Z_n^\psi\left(e^{-\hl_n^u}\eps_n^f\tau_n\right) + e^{-\hl_n^u}(1+\gamma_n)\hZ_n^f\left((1+e^{-\hl_n^u}\eps_n^f(1+\gamma_n))\tau_n\right)
\]
and note that $\eps_n^\sigma=0$ if $\tau_n=0$, that is, if we consider admissible manifolds passing \emph{through} $0$, not just near it.  We will require the following recursive bounds on the parameters:
\begin{align}
\label{eqn:rn}
r_{n+1} &\leq e^{\hl_n^u} r_n - \eps_n^f \tau_n, \\
\label{eqn:taun}
\tau_{n+1} &\geq 
e^{\cl_n^s} \tau_n, \\
\label{eqn:thetan}
\sigma_{n+1} &\geq e^{\lambda_n^s - \hl_n^u} \sigma_n + \eps_n^\sigma, \\
\label{eqn:gn}
\gamma_{n+1} &\geq \min \left(e^{\hl_n^s - \hl_n^u} \gamma_n,\, \sigma_{n+1} + Z_{n+1}^\psi(r_{n+1}) \right).
\end{align}


\begin{theorem}\label{thm:HP1}
If the sequence of maps $\{f_n\mid 0\leq n<N\}$ satisfies \Cp\ and \eqref{eqn:rnbn}--\eqref{eqn:gn} hold, then the following are true.
\begin{thmenum}
\item\label{1gt}
The graph transform $G_n\colon \CCC_n'\to\CCC_{n+1}'$ is well-defined for every $0\leq n<N$.
\item\label{1dyn}
Given $\psi_0\in \CCC_0'$, the $C^1$ functions $\psi_n = \GGG_n\psi_0\colon B_n^u(r_n)\to E_n^s$ have the following property: if $x,y\in (\graph\psi_m) \cap \Omega_m^n$ for some $0\leq m\leq n$, then
\begin{equation}\label{eqn:dFn}
\|F_{m,n}(x) - F_{m,n}(y)\| \geq \ex{\chi_m + \cdots + \chi_{n-1}} \|x - y\|.
\end{equation}
\item \label{1nearby}
Fix $(v_0,w_0)\in \Omega_0^n$ and let $(v_n,w_n) = F_{0,n}(v_0,w_0)$.  Then
\begin{equation}\label{eqn:attraction}
\|w_n - \psi_n(v_n)\| \leq e^{\sum_{k=0}^{n-1}\cl_k^s} \|w_0 - \psi_0(v_0)\|,
\end{equation}
and if $(v_n',w_n')$ is another trajectory such that
\[
\|w_0' - w_0\| \leq \gamma_0 \|v_0' - v_0\|,
\]
then
\begin{equation}\label{eqn:incone}
\|w_n' - w_n\| \leq \gamma_n \|v_n' - v_n\|
\end{equation}
for all $0\leq n<N$.  Moreover, we have
\begin{equation}
\label{eqn:vnvn}
\|v_n - v_n'\| \geq e^{\sum_{k=0}^{n-1}\hl_k^u} \|v_0 - v_0'\|.
\end{equation}
\item\label{1contract}
Define $r_n^{(k)}$ for $0\leq k\leq n$ by $r_n^{(n)} = r_n$ and $r_n^{(k)} = e^{-\hl_k^u} (r_n^{(k+1)} + \eps_{k}^f \tau_{k})$.  Then
given $\psi_0,\ph_0\in \CCC_0'$ and writing $\hat r_n = r_n^{(0)}$, we have,
\begin{equation}\label{eqn:contraction}
\|\psi_n - \ph_n\|_{C^0} \leq \ex{\cl_0^s + \cl_1^s + \cdots + \cl_{n-1}^s} \left\|(\psi_0 - \ph_0)|_{B_0^u(\hat{r}_n)}\right\|_{C^0}.
\end{equation}
\end{thmenum}
\end{theorem}

\begin{remark}
Theorem~\ref{thm:HP1} is valid even without any assumptions on the existence of genuine contraction or expansion in $E_n^s$ and $E_n^u$, or any domination.  It gives information on admissible manifolds based on information from the tangent space, without any requirement of uniform or non-uniform hyperbolicity.  
\end{remark}

\subsection{Preliminaries for the proof}\label{sec:C1setup}

As usual, we use coordinates on $\Omega_n \subset \RR^d$ given by the decomposition $\RR^d = E_n^u \oplus E_n^s$.  Thus given $v\in E_n^u$ and $w\in E_n^s$, we write $(v,w) = v+w \in \RR^d$.  We let $\di_1$ denote the partial derivative with respect to $v$, and $\di_2$ the partial derivative with respect to $w$.

Consider the error function $s_n\colon \Omega_n \to \RR^d$ given by $s_n = f_n - Df_n(0)$; then $\di_i s_n$ has $Z_n^f$ as a modulus of continuity. 
Writing $A_n = Df(0)|_{E_n^u}$ and $B_n = Df(0)|_{E_n^s}$, we see that $Df_n(0)$ takes the diagonal form
\[
(v,w) \mapsto (A_n v, B_n w).
\]
Similarly, we write
\begin{equation}\label{eqn:rngnhn}
s_n(v,w) = (g_n(v,w), h_n(v,w)).
\end{equation}
We want to use $Z_n^f$ to describe a modulus of continuity for $Dg_n$ and $Dh_n$; 
here the angle $\theta_n$ between $E_n^u$ and $E_n^s$ becomes important.  Indeed, it is easy to see that if $a,b,c$ are sides of a triangle and $\theta$ is the angle between $a$ and $b$, then $c \geq a\sin \theta$ (and also $c \geq b\sin \theta$).  Given $x,y\in \RR^d$, we apply this with $a=\di_i g_n(x) - \di_i g_n(y)$, $b=\di_i h_n(x) - \di_i h_n(y)$, $c=\di_i s_n(x) - \di_i s_n(y)$, and $\theta = \theta_{n+1}$ to obtain
\begin{multline*}
\|\di_i g_n(x) - \di_i g_n(y)\| \leq \|\di_i s_n(x) - \di_i s_n(y)\| (\sin \theta_{n+1})^{-1} \\
\leq Z_n^f(\|x-y\|)(\sin\theta_{n+1})^{-1} = \hZ_n^f(\|x-y\|),
\end{multline*}
and similarly for $\di_i h_n$ (see~\eqref{eqn:hZ} for the last step).  This shows that $\hZ_n^f$ is a modulus of continuity for both $\di_i g_n$ and $\di_i h_n$.  In particular, we see from~\eqref{eqn:Un} and~\eqref{eqn:rnbn} that both $g_n$ and $h_n$ are Lipschitz with constant $\eps_n^f$, so that
\begin{equation}\label{eqn:ghLip}
\begin{aligned}
\|g_n(x) - g_n(y)\| &\leq \eps_n^f \|x - y\|, \qquad\qquad \|g_n(x)\| \leq \eps_n^f \|x\|, \\
\|h_n(x) - h_n(y)\| &\leq \eps_n^f \|x - y\|, \qquad\qquad \|h_n(x)\| \leq \eps_n^f \|x\|
\end{aligned}
\end{equation}
for every $x,y\in \Omega_n$, where the second inequality on both lines uses the fact that $g_n(0)=h_n(0)=0$.



\subsection{Defining the graph transform}

Given $\psi_n \in \CCC_n'$, we use the coordinates provided by $E_{n+1}^u$ and $E_{n+1}^s$ to investigate the manifold $f_n(\graph \psi_n)$.  Our initial goal is to show that $f_n(\graph \psi_n \cap \Omega_n^{n+1})$ is the graph of a function $\psi_{n+1}\colon B_{n+1}^u(r_{n+1}) \to E_{n+1}^s$.

To this end, to every $v\in B_n^u(r_n)$ we associate $\bar{v} \in E_{n+1}^u$ and $\bar{\psi} \in E_{n+1}^s$ such that
\begin{equation}\label{eqn:barpsiv}
(\bar{v},\bar{\psi}) = f_n(v,\psi(v)) = (A_n v + g_n(v,\psi(v)), B_n \psi(v) + h_n(v,\psi(v))).
\end{equation}
We must show that the image of the map $v\mapsto \bar{v}$ contains the set $B_{n+1}^u(r_{n+1})$ and that the inverse map $\bar{v}\mapsto v$ can be properly defined here; then we can compose this inverse map with the map $v\mapsto \bar{\psi}$ to obtain the desired map $\bar{v} \mapsto \psi_{n+1}(\bar{v}) = \bar{\psi}(v(\bar{v}))$.  Then we will show that the new map has $\|\bar{\psi}(0)\| \leq \tau_{n+1}$.

Finally, after computing $\frac{\di\bar{v}}{\di v}$ and $\frac{\di\bar{\psi}}{\di v}$, we must use these to show that $\|D\psi_{n+1}(0)\| \leq \sigma_{n+1}$, that $\|D\psi_{n+1}\| \leq \gamma_{n+1}$, and that $Z_{n+1}^\psi$ is a modulus of continuity for $D\psi_{n+1}$.


From now on, to simplify notation, we write $g_n(v) = g_n(v,\psi_n(v))$ and $h_n(v) = h_n(v,\psi_n(v))$.  We also drop the explicit dependence on $n$ for the maps $A,B,g,h,\psi$, whenever it does not cause confusion.  (We will retain the subscript for the various parameters.) 
Then~\eqref{eqn:barpsiv} may be rewritten as the following pair of equations:
\begin{align}
\label{eqn:imp1}
\bar{v} &= A v + g(v), \\
\label{eqn:imp2}
\bar{\psi} &= B \psi(v) + h(v).
\end{align}
Using the fact that $\hZ_n^f$ is a modulus of continuity for $\di_i g_n$, together with the estimates $\|D\psi(v)\| \leq \gamma_n$ and $\|\psi(v)\| \leq \tau_n + \gamma_n \|v\|$, we see that
\begin{equation}\label{eqn:gerror1}
\begin{aligned}
\|D g(v)\| &= \|\di_1 g_n(v,\psi(v)) + \di_2 g_n(v,\psi(v)) \circ D \psi(v)\| \\
&\leq  (1+\gamma_n) \hZ_n^f(\|(v,\psi(v))\|) \\
&\leq (1+\gamma_n) \hZ_n^f(\tau_n + (1+\gamma_n)r_n) = (1+\gamma_n) \eps_n^f,
\end{aligned}
\end{equation}
and similarly,
\begin{equation}\label{eqn:herror1}
\|D h(v)\| \leq (1+\gamma_n) \eps_n^f.
\end{equation}
In particular, we see from~\eqref{eqn:hlnu} that 
\begin{equation}\label{eqn:barvexpands}
\|(A + Dg(v))^{-1}\|^{-1} \geq e^{\lambda_n^u} - (1+\gamma_n)\eps_n^f = e^{\hl_n^u}.
\end{equation}
If follows that given $v_1,v_2 \in B_n^u(r_n)$, we have
\begin{equation}\label{eqn:barv2}
\|\bar{v}_1 - \bar{v}_2\| \geq e^{\hl_n^u} \|v_1 - v_2\|. 
\end{equation}
In particular, the map $v\mapsto \bar{v}$ is one-to-one on $B_n^u(r_n)$.  Using~\eqref{eqn:imp1} and~\eqref{eqn:ghLip}, we have $\|\bar{v}\| = \|g_n(0,\psi(v))\| \leq \tau_n \eps_n^f$ when $v=0$, and it follows from~\eqref{eqn:rn} and~\eqref{eqn:barvexpands} that the image of $B_n^u(r_n)$ under the map $v\mapsto \bar{v}$ contains $B_{n+1}^u(r_{n+1})$.  In particular,~\eqref{eqn:imp1} and~\eqref{eqn:imp2} determine a well-defined function $\bar{\psi}\colon B_{n+1}^u(r_{n+1}) \to E_{n+1}^s$.



To compute $\bar\psi(0)$, we let $v_1=0$ and take $v_2$ to be such that $\bar{v}_2 = 0$.  Then~\eqref{eqn:ghLip} gives $\|\bar{v}_1\| \leq \eps_n^f \tau_n$, whence we use~\eqref{eqn:barv2} to deduce that
\begin{equation}\label{eqn:v2}
\|v_2\| \leq e^{-\hl_n^u} \|\bar{v}_1\| \leq e^{-\hl_n^u} \eps_n^f \tau_n.
\end{equation}
Together with~\eqref{eqn:imp2} and~\eqref{eqn:taun}, this implies that
\begin{align*}
\|\bar\psi(0)\| &\leq e^{\lambda_n^s} \|\psi(v_2)\| + \|h(v_2)\| \\
&\leq e^{\lambda_n^s} (\tau_n + \gamma_n \|v_2\|) + \eps_n^f \|v_2\| \\
&\leq \left( e^{\lambda_n^s} + \gamma_n \eps_n^f e^{\lambda_n^s - \hl_n^u} + (\eps_n^f)^2 e^{-\hl_n^u} \right) \tau_n \leq \tau_{n+1}.
\end{align*}

\subsection{Regularity properties of $\psi_{n+1}$}

We now estimate the regularity properties of the map $\bar{\psi}$.  Differentiating~\eqref{eqn:imp1} and~\eqref{eqn:imp2} gives
\begin{align*}
\frac{d\bar{v}}{d v} &= A + D g(v), \\
\frac{d\bar{\psi}}{d v} &= B\circ D\psi(v) + Dh(v).
\end{align*}
Write $\hA(v) = \frac{d\bar{v}}{d v} = A + Dg(v)$; we saw in~\eqref{eqn:barvexpands} that $\|\hA(v)^{-1}\|^{-1} \geq e^{\hl_n^u}$ for every $v\in B_n^u(r_n)$.  Now using the chain rule, we conclude that
\begin{equation}\label{eqn:barpsi'}
\begin{aligned}
D\bar{\psi}(\bar{v}) &= (B\circ D\psi(v) + Dh(v)) \circ (A+Dg(v))^{-1}, \\
&= (B\circ D\psi(v) + Dh(v)) \circ \hA(v)^{-1}.
\end{aligned}
\end{equation}
Recalling that $\log \|B\| \leq \lambda_n^s$ and $\|D\psi(v)\| \leq \gamma_n$, we let $v$ be such that $\bar{v}=0$, and use~\eqref{eqn:v2},~\eqref{eqn:gerror1}, and~\eqref{eqn:barpsi'} to estimate $\|D\bar\psi(0)\|$:
\begin{align*}
\|D\bar\psi(0)\| &\leq \|\hat{A}(v)^{-1}\|(\|B\| \|D\psi(v)\| + \|Dh(v)\|) \\
&\leq e^{-\hl_n^u} \bigg( e^{\lambda_n^s} \Big(\sigma_n + Z_n^\psi\big(e^{\hl_n^u} \eps_n^f \tau_n\big)\Big) \\
&\qquad\qquad + (1+\gamma_n) \hZ_n^f\big(\tau_n + (1+\gamma_n) e^{-\hl_n^u} \eps_n^f \tau_n\big) \bigg).
\end{align*}
Recalling the definition of $\eps_n^\sigma$ before Theorem~\ref{thm:HP1}, this shows that $\|D\bar\psi(0)\|\leq \sigma_{n+1}$ as long as $\sigma_{n+1}$ satisfies~\eqref{eqn:thetan}.

Now we use~\eqref{eqn:herror1}, \eqref{eqn:barvexpands}, and~\eqref{eqn:barpsi'} to estimate $\|D\bar\psi\|$, requiring only that $\|v\|\leq r_n$:
\[
\|D\bar\psi(\bar v)\| \leq \big(e^{\lambda_n^s} \gamma_n + (1+\gamma_n)\eps_n^f\big) e^{-\hl_n^u} 
\leq e^{\hl_n^s - \hl_n^u} \gamma_n,
\]
where the last step uses~\eqref{eqn:hlns}.

Observe that~\eqref{eqn:gn} may be satisfied in one of two ways: either we have $\gamma_{n+1} \geq e^{\hl_n^s - \hl_n^u} \gamma_n$, or we have $\gamma_{n+1} \geq \sigma_{n+1} + Z_{n+1}^\psi(r_{n+1})$.  In the first case, the inequality $\|D\psi_{n+1}\| \leq \gamma_{n+1}$ follows from the argument above.  In the second case, this inequality follows from the fact that $Z_{n+1}^\psi$ is a modulus of continuity for $D\psi_{n+1}$, which we now prove.

\begin{remark}
We will need to use the second case in the proof of Theorem~\ref{thm:parameters}.
\end{remark}

To show that $Z_{n+1}^\psi$ is a modulus of continuity for $D\psi_{n+1}$, we must estimate the quantities $\|D\bar{\psi}(\bar{v}_1) - D\bar{\psi}(\bar{v}_2)\|$ and $\|\bar{v}_1 - \bar{v}_2\|$.  First we observe that
\begin{multline}\label{eqn:Dbp}
D\bar{\psi}(\bar{v}_1) - D\bar{\psi}(\bar{v}_2) = 
(B\circ D\psi(v_1) + Dh(v_1)) \circ \hA(v_1)^{-1} \\
- (B\circ D\psi(v_2) + Dh(v_2)) \circ \hA(v_2)^{-1}.
\end{multline}
Furthermore, it follows from the definition of $\hA(v)$ that
\[
\hA(v_1) = \hA(v_2) + Dg(v_1) - Dg(v_2);
\]
composing on the left by $\hA(v_2)^{-1}$ and on the right by $\hA(v_1)^{-1}$ yields
\[
\hA(v_2)^{-1} = \hA(v_1)^{-1} + \hA(v_2)^{-1} \circ (Dg(v_1) - Dg(v_2)) \circ \hA(v_1)^{-1}.
\]
Using this in~\eqref{eqn:Dbp} gives
\begin{multline*}
D\bar{\psi}(\bar{v}_1) - D\bar{\psi}(\bar{v}_2)
= \big(B \circ(D\psi(v_1) - D\psi(v_2) ) + (Dh(v_1)) - Dh(v_2)) \\
+ (B\circ D\psi(v_2) + Dh(v_2)) \circ \hA(v_2)^{-1} \circ (Dg(v_1) - Dg(v_2))\big) \circ \hA(v_1)^{-1}.
\end{multline*}
Writing $t=\|v_1 - v_2\|$, this leads to the following estimate:
\begin{equation}\label{eqn:deltaDbar}
\begin{aligned}
\|D\bar{\psi}(\bar{v}_1) &- D\bar{\psi}(\bar{v}_2)\|
\leq \big(\|B\| Z_n^\psi(t) + \hZ_n^f(t) \\
&\qquad + (\|B\| \|D\psi\|  + \|Dh\| ) \cdot \|\hA(v_1)^{-1}\| \cdot \hZ_n^f(t) \big) \|\hA(v_2)^{-1} \| \\
&\leq \big(e^{\lambda_n^s} Z_n^\psi(t) + \hZ_n^f(t)
+ (e^{\lambda_n^s} \gamma_n  + \hZ_n^f(t) ) e^{-\hl_n^u} \hZ_n^f(t) \big) e^{-\hl_n^u}.
\end{aligned}
\end{equation}
Now~\eqref{eqn:barvexpands},~\eqref{eqn:deltaDbar}, and~\eqref{eqn:kn} show that $Z_{n+1}^\psi$ is a modulus of continuity for $D\psi_{n+1}$.

It follows from the definition of $\psi$ that $\graph\psi_{n+1} = f_n(\graph\psi_n \cap \Omega_n^{n+1})$, and thus induction shows that $\graph\psi_n = F_{0,n}(\graph\psi_0 \cap \Omega_0^n)$ for all $n$, which completes the proof of~\ref{1gt}

\subsection{Dynamics of $f_n\colon \graph \psi_n \to \graph \psi_{n+1}$}

To prove~\ref{1dyn}, we must establish~\eqref{eqn:dFn} by estimating the expansion of the map $f_n$ from $\graph\psi_n$ to $\graph\psi_{n+1}$.  In particular, we must show that given $x,y\in \graph(\psi_n) \cap f_n^{-1} \graph(\psi_{n+1})$, we have
\begin{equation}\label{eqn:expand0}
\|f_n(x) - f_n(y)\| \geq e^{\chi_n} \|x-y\|.
\end{equation}
Using the definition of $\chi_n$ in~\eqref{eqn:chin}, this is equivalent to proving both of the following inequalities:
\begin{align}
\label{eqn:expand1}
\|f_n(x) - f_n(y)\| &\geq \frac{1-\gamma_{n+1}}{1+\gamma_n} e^{\hl_n^u} \|x-y\|, \\
\label{eqn:expand2}
\|f_n(x) - f_n(y)\| &\geq \frac{\sin\theta_{n+1}}{1+\gamma_n} e^{\hl_n^u} \|x-y\|.
\end{align}

Now suppose $v_1,v_2\in B_n^u(r_n)$ are such that $\bar v_1, \bar v_2$ lie in $B_{n+1}^u(r_{n+1})$.  To prove~\eqref{eqn:expand1}, we use the estimate
\[
\|\bar\psi(\bar v_1) - \bar\psi(\bar v_2)\| \leq \gamma_{n+1} \|\bar v_1 - \bar v_2\|
\]
and observe that
\begin{align*}
\|(\bar v_1,\bar{\psi}(\bar v_1)) - (\bar v_2,\bar{\psi}(\bar v_2))\| 
&= \|(\bar v_1 - \bar{v_2},\bar{\psi}(\bar v_1) -\bar{\psi}(\bar v_2))\|  \\
&\geq (1-\gamma_{n+1}) \|\bar v_1 - \bar v_2\| \\
&\geq (1-\gamma_{n+1}) e^{\hl_n^u} \|v_1 - v_2\| \\
&\geq \frac{1-\gamma_{n+1}}{1+\gamma_n} e^{\hl_n^u} \|(v_1,\psi(v_1)) - (v_2, \psi(v_2))\|.
\end{align*}
For~\eqref{eqn:expand2}, we use the triangle estimates discussed following~\eqref{eqn:rngnhn} and obtain
\begin{align*}
\|(\bar v_1,\bar{\psi}(\bar v_1)) - (\bar v_2,\bar{\psi}(\bar v_2))\| 
&= \|(\bar v_1 - \bar v_2,\bar{\psi}(\bar v_1)) -\bar{\psi}(\bar v_2))\|  \\
&\geq \sin \theta_{n+1} \|\bar v_1 - \bar v_2\| \\
&\geq \sin \theta_{n+1} e^{\hl_n^u} \|v_1 - v_2\| \\
&\geq \frac{\sin \theta_{n+1}}{1+\gamma_n} e^{\hl_n^u} \|(v_1,\psi(v_1)) - (v_2, \psi(v_2))\|.
\end{align*}
Together these establish~\eqref{eqn:expand0}, and~\eqref{eqn:dFn} follows by induction, completing the proof of~\ref{1dyn}

\subsection{Contraction properties of the graph transform}\label{sec:contraction}

First we observe that part \ref{1contract} of the theorem follows from part \ref{1nearby}  Indeed, it follows from~\eqref{eqn:barvexpands} and the remarks after~\eqref{eqn:barv2} that to compare $\ph_n,\psi_n$ on $B^u(r_n)$, it suffices to compare $\ph_0,\psi_0$ on $B^u(\hat r_n)$, and then~\eqref{eqn:attraction} establishes the rest of \ref{1contract}

For part \ref{1nearby}, we see that~\eqref{eqn:vnvn} comes from exacly the same argument as~\eqref{eqn:barv2}, where we need only replace the function $\psi$ from that argument by another function in $\CCC_n'$ whose graph contains both $(v_n,w_n)$ and $(v_n',w_n')$ -- this is possible by~\eqref{eqn:incone}.

Thus it only remains to prove~\eqref{eqn:attraction}.  To this end, write $(v_1,w_1) = (v_n,w_n)$ and $(\hat v_1, \hat w_1) = (v_{n+1}, w_{n+1})$.  Let $(v_1,\psi_1)$ be the point on $\graph\psi$ with the same $u$-coordinate as $(v_1,w_1)$, and let $(\bar v_2, \bar \psi_2)$ be the point on $\graph\bar\psi$ with the same $u$-coordinate as $(\hat v_1,\hat w_1)$, so that $\bar v_2 = \hat v_1$.  Let $(v_2,\psi_2) = f_n^{-1}(\bar v_2, \bar\psi_2)$.

Now we have
\begin{equation}\label{eqn:vsws}
\begin{alignedat}{3}
\bar v_1 &= A_n v_1 + g_n(v_1,\psi_1),\qquad & \bar\psi_1 &= B_n\psi_1 + h_n(v_1,\psi_1),\\
\bar v_2 &= A_n v_2 + g_n(v_2,\psi_2),\qquad & \bar\psi_2 &= B_n\psi_2 + h_n(v_2,\psi_2),\\
\hat v_1 &= A_n v_1 + g_n(v_1,w_1), & \hat w_1 &= B_n w_1 + h_n(v_1,w_1).
\end{alignedat}
\end{equation}
We must estimate $\|\hat w_1 - \bar\psi_2\|$ in terms of $\|w_1 - \psi_1\|$.  Using~\eqref{eqn:ghLip} and~\eqref{eqn:vsws}, we have
\begin{equation}\label{eqn:hatbar}
\|\hat w_1 - \bar\psi_2\| \leq e^{\lambda_n^s} \|w_1 - \psi_2\| + \eps_n^f (\|v_1 - v_2\| + \|w_1 - \psi_2\|).
\end{equation}
Furthermore, we have $\|w_1 - \psi_2\| \leq \|w_1 - \psi_1\| + \|\psi_1 - \psi_2\|$, and we can use~\eqref{eqn:barv2}, \eqref{eqn:ghLip}, and \eqref{eqn:vsws} to obtain
\[
\|v_1 - v_2\| \leq e^{-\hl_n^u} \|\bar v_1 - \hat v_1\| \leq e^{-\hl_n^u} \eps_n^f \|w_1 - \psi_1\|.
\]
Together with~\eqref{eqn:hatbar} and the hypothesis on $\|D\psi_n\|$, this yields
\begin{align*}
\|\hat w_1 - \bar\psi_2\| &\leq
(e^{\lambda_n^s} + \eps_n^f) (\|w_1 - \psi_1\| + \|\psi_1 - \psi_2\|) + \eps_n^f \|v_1 - v_2\| \\
&\leq (e^{\lambda_n^s} + \eps_n^f)(1 + \gamma_n e^{-\hl_n^u} \eps_n^f) \|w_1 - \psi_1\| + e^{-\hl_n^u} (\eps_n^f)^2 \|w_1 - \psi_1\| \\
&= \left( e^{\lambda_n^s} + (1 + \gamma_n e^{\lambda_n^s - \hl_n^u}) \eps_n^f + (1+\gamma_n) e^{-\hl_n^u} (\eps_n^f)^2\right) \|w_1 - \psi_1\| \\
&= e^{\cl_n^s} \|w_1 - \psi_1\|,
\end{align*}
where the last equality uses the definition in~\eqref{eqn:clns}.  This completes the proof of \ref{1nearby}

\section{General results on unstable manifolds}\label{sec:C1unstable}

Now we consider a sequence $\{f_n \mid n\leq 0\}$ of $C^1$ maps and produce unstable manifolds by applying Theorem~\ref{thm:HP1} to the finite sequences $\{f_k \mid n\leq k<0\}$.\footnote{We could consider $\{f_n \mid n\geq 0\}$ and obtain results on stable manifolds instead of unstable manifolds, but the notation and bounds laid out in Section~\ref{sec:HP1} are more suited to describing unstable manifolds for a sequence $\{f_n \mid n\leq 0\}$.}

The theorem below relies on having a sequence $Z_n^\psi$ of moduli of continuity satisfying~\eqref{eqn:kn}: for now we assume that such a sequence has already been found, and in Proposition~\ref{prop:checkmoduli} below we give conditions on $\hat Z_n^f, \hl_n^u, \hl_n^s,\gamma_n$ that guarantee the existence of such $Z_n^\psi$.


\begin{theorem}\label{thm:HP3}
Let $\{f_n\mid n\leq 0\}$ satisfy \Cp\ and suppose $r_n,\gamma_n,Z_n^\psi$ are such that~\eqref{eqn:rnbn}--\eqref{eqn:gn} hold with $\sigma_n=\tau_n=0$.  Then the following are true.
\begin{thmenum}
\item\label{inv}
Writing $\CCC_n' = \CCC_n'(r_n,0,0,\gamma_n, Z_n^\psi)$, there exists $\psi_n\in \CCC_n'$ such that $G_n\psi_n=\psi_{n+1}$ for all $n<0$.
\item\label{inv-exp}
If $x,y\in (\graph\psi_m) \cap \Omega_m^n$ for some $n>m$, then
\begin{equation}\label{eqn:dFn2}
\|F_{m,n}(x) - F_{m,n}(y)\| \geq \ex{\chi_m + \cdots + \chi_{n-1}} \|x - y\|.
\end{equation}
\item\label{unique}
If we have
\begin{equation}\label{eqn:backwards}
\llim_{n\to -\infty} \log \gamma_n + \sum_{n\leq k < 0}  \left(\cl_k^s - \hl_k^u\right) = -\infty,
\end{equation}
then $\psi_n$ is the unique member of $\CCC_n'$ satisfying~\ref{inv}
\item\label{dyn-char}
If $x\in \Omega_n$ is such that $x_m := F_{m,n}^{-1}(x_n) \in \Omega_m$ for every $m\leq n$ and
\begin{equation}\label{eqn:backcontr}
\llim_{m\to-\infty} (1+\gamma_m) \|x_m\| \ex{-\sum_{k=m}^{n-1} \cl_k^s} = 0,
\end{equation}
then $x\in \graph\psi_n$.
\item\label{dyn-char-2}
If $\gamma_m$ is bounded above and $\lambda_k$ is a sequence such that $\lambda_k \geq \cl_k^s + t(\hl_k^u - \cl_k^s)$ for some fixed $t$ (independent of $k$), and if $x$ is such that there exists $C$ with
\begin{equation}\label{eqn:backcontrb}
\|x_m\| \leq C \ex{-\sum_{k=m}^{n-1} \lambda_k}
\end{equation}
for every $m$, then~\eqref{eqn:backcontr} holds and $x\in \graph \psi_n$.
\end{thmenum}
\end{theorem}

\begin{proof}
Theorem~\ref{thm:HP1} shows that the graph transform $G_n\colon \CCC_n'\to\CCC_{n+1}'$ is well-defined for all $n<0$.  To show existence of the family $\psi_n$, we define for each $k<0$ a family $\Psi^k = (\psi_n^k)_{n<0} \in \prod_{n<0} \CCC_n'$ by
\begin{equation}\label{eqn:Psik}
\psi_n^k = \begin{cases}
           \mathbf{0} & n\leq k, \\
           G_{n-1} \psi_{n-1}^k & n > k,
           \end{cases}
\end{equation}
where $\mathbf{0}$ is the zero function.  By the Arzela--Ascoli theorem, $\CCC_n'$ is $C^1$-compact because $\{ D\psi \mid \psi \in \CCC_n'\}$ is an equicontinuous and bounded family of functions.  
Thus by Tychonoff's theorem, $\prod_{n<0} \CCC_n'$ is compact in the product topology.  In particular, there exists $k_j\to -\infty$ such that $\Psi^{k_j} \to \Psi = (\psi_n)_{n<0} \in \prod \CCC_n'$, and this sequence $\psi_n\in \CCC_n'$ satisfies $G_n\psi_n = \psi_{n+1}$ by the second part of~\eqref{eqn:Psik}.  This proves Part \ref{inv}

Part \ref{inv-exp} follows directly from Part \ref{1dyn} of Theorem~\ref{thm:HP1}.  To prove the claim of uniqueness in Part \ref{unique}, we again consider the sequence $\Psi^k$ defined in~\eqref{eqn:Psik} and estimate $\|\psi_n^j - \psi_n^k\|$ using Part \ref{1contract} of Theorem~\ref{thm:HP1}.  Note that because $\tau_\ell=0$ for all $\ell$, we have $r_n^{(\ell)} = \exp(-\sum_{i=\ell}^{n-1} \hl_i^u)$.  Now take $m<n$ to be large negative and $j,k\geq m$; then~\eqref{eqn:contraction} gives
\begin{align*}
\|\psi_n^j - \psi_n^k\|_{C^0} &\leq \exp\left( \sum_{i=m}^{n-1} \cl_i^s \right) \left\|(\psi_m^j - \psi_m^k)|_{B_m^u(r_n^{(m)})}\right\| \\
&\leq 2\gamma_{n-m} \ex{\sum_{i=n-m}^{n-1} \cl_i^s - \hl_i^u}.
\end{align*}
Together with~\eqref{eqn:backwards}, this implies that the sequence $\{ \psi_n^{n-i} \mid i\in \NN\}$ is Cauchy in the uniform metric, and hence there exists a continuous function $\psi_n \colon B_n^u(r_n) \to E_n^s$ such that $\lim_{k\to-\infty} \psi_n^k = \psi_n$.  In particular, once~\eqref{eqn:backwards} holds there is no need to pass to a subsequence $k_j$ to obtain convergence.



To prove Part \ref{dyn-char}, we apply~\eqref{eqn:attraction} to the sequence of points $x_m$.  For every $m \leq n$, we get
\begin{equation}\label{eqn:stablepart}
\|w_n - \psi_n(v_n)\| \leq \ex{-\sum_{k=m}^{n-1} \cl_k^s} \|w_m - \psi_m(v_m) \|.
\end{equation}
Using the fact that $\|w_m - \psi_m(v_m) \| \leq (1+\gamma_m) \|x_m\|$ together with~\eqref{eqn:backcontr}, the right hand side of~\eqref{eqn:stablepart} becomes arbitrarily small as $m\to-\infty$, and it follows that $x^s = \psi_n(x^u)$, or in other words, $x\in \graph\psi_n$.

For Part \ref{dyn-char-2}, it suffices to observe that~\eqref{eqn:backwards} and~\eqref{eqn:backcontrb} imply~\eqref{eqn:backcontr} when $\gamma_m$ is bounded above.
\end{proof}

\begin{proposition}\label{prop:checkmoduli}
Given~\eqref{eqn:hlnu}--\eqref{eqn:gn}, suppose that the sum
\begin{multline}\label{eqn:Znpsisum}
Z_n^\psi(t) = \sum_{k< n}  \ex{-\hl_k^u + \sum_{k < j < n} (\lambda_j^s - \hl_j^u)} \\
\cdot \left( 1 + e^{\hl_k^s - \hl_k^u} \gamma_k\right) \hZ_k^f\left( t\ex{-\sum_{k\leq j < n} \hl_j^u} \right)
\end{multline}
converges when $n=0$ for all $t\in (0,r_0)$, and that $\lim_{t\to 0} Z_0^\psi(t) = 0$.  Then $Z_n^\psi$ is a sequence of moduli of continuity satisfying~\eqref{eqn:kn}.
\end{proposition}
\begin{proof}
It follows from~\eqref{eqn:Znpsisum} that for all $n<0$, we have
\begin{align*}
\big( e^{\lambda_n^s} &Z_n^\psi(t) + (1+e^{\hl_n^s-\hl_n^u} \gamma_n ) \hZ_n^f(t) \big) e^{-\hl_n^u} \\
&=  e^{\lambda_n^s - \hl_n^u} \sum_{k< n} e^{-\hl_k^u + \sum_{k < j < n} (\lambda_j^s - \hl_j^u)}\left( 1 + e^{\hl_k^s - \hl_k^u} \gamma_k\right)
\hZ_k^f\left( te^{-\sum_{k\leq j < n} \hl_j^u} \right) \\
&\qquad \qquad + e^{-\hl_n^u} (1+e^{\hl_n^s-\hl_n^u} \gamma_n ) \hZ_n^f(t) \\
&= \sum_{k\leq n} e^{-\hl_k^u + \sum_{k<j\leq n} (\lambda_j^s - \hl_j^u)} \left( 1 + e^{\hl_k^s - \hl_k^u} \gamma_k \right)
\hZ_k^f \left(te^{\hl_n^u} e^{-\sum_{k\leq j\leq n} \hl_j^u} \right) \\
&= Z_{n+1}^\psi(t e^{\hl_n^u}).
\end{align*}
This shows that~\eqref{eqn:kn} holds, and solving for $Z_n^\psi$ shows that it is a legitimate modulus of continuity function for each $n$.
\end{proof}

\section{Proof of results in Section~\ref{sec:C1+}}

\subsection{Proof of Theorem~\ref{thm:parameters}}\label{sec:param-pfs}

We now prove Theorem~\ref{thm:parameters} using Theorem~\ref{thm:HP1}.  We begin by estimating the quantities in~\eqref{eqn:rnbn}--\eqref{eqn:rhon} using~\eqref{eqn:rec-r}--\eqref{eqn:bd-s} and then showing that for any $\delta$ and $L$, we can choose $\xi$ and $\bar\gamma$ such that~\eqref{eqn:kn}--\eqref{eqn:gn} are satisfied.

Using~\eqref{eqn:rnbn},~\eqref{eqn:bd-t}, and~\eqref{eqn:bd-r}, we have (taking $\bar\gamma\leq 1$)
\begin{equation}\label{eqn:enf}
\eps_n^f =\beta_n(\tau_n+r_n(1+\bar\gamma))^\alpha \leq \beta_n r_n^\alpha (2+\bar\gamma) \leq 3\xi.
\end{equation}
Together with~\eqref{eqn:hlnu} this gives $\eps_n^u \leq 6\xi$.  Fix $\zeta>0$ such that $(2+\alpha)\zeta < \delta$.  
By the assumption that $\lambda_n^u \geq -L$, we can choose $\xi$ sufficiently small that
\begin{equation}\label{eqn:enu}
e^{\hl_n^u} = e^{\lambda_n^u} - \eps_n^u \geq e^{\lambda_n^u} - 6\xi \geq e^{\lambda_n^u - \zeta}.
\end{equation}
Let $L_1=e^{L+\zeta}$, so that for all $n$ we have
\begin{equation}\label{eqn:L1}
e^{-\hl_n^u} \leq L_1 \text{ and } e^{\lambda_n^s - \hl_n^u} \leq L_1.
\end{equation}
Now choose $\bar\gamma$ sufficiently small so that in~\eqref{eqn:chin} we have
\begin{equation}\label{eqn:enchi}
\chi_n \geq \lambda_n^u + \log\left(\frac{1-\bar\gamma}{1+\bar\gamma}\right) \geq \lambda_n^u - 2\zeta.
\end{equation}
From~\eqref{eqn:clns},~\eqref{eqn:enf}, and~\eqref{eqn:L1}, we have
\[
\hat\eps_n \leq (1+L_1\bar\gamma)(3\xi) + (1+\bar\gamma)L_1(3\xi)^2,
\]
and so, using the fact that $\lambda_n^s\geq -L$ and decreasing $\xi$ if necessary,~\eqref{eqn:clns} gives
\begin{equation}\label{eqn:clns2}
\cl_n^s \geq \lambda_n^s - \zeta.
\end{equation}
Turning to~\eqref{eqn:rhon} and~\eqref{eqn:kn}, we see from~\eqref{eqn:L1},~\eqref{eqn:bd-r}, and~\eqref{eqn:bd-b} that
\begin{align*}
\rho_n(t) &\leq L_1(1+L_1\bar\gamma)\beta_n t^\alpha + L_1^2 \beta_n^2 t^{2\alpha} \\
&\leq \big(L_1(1+L_1\bar\gamma) \xi\kappa_n + L_1^2 \xi \big) t^\alpha
\end{align*}
for every $t\in [0, r_n]$.  We use this to prove~\eqref{eqn:kn}.  Indeed, for the moduli of continuity $Z_n^\psi(t) = \kappa_n t^\alpha$, the quantity on the right side of~\eqref{eqn:kn} is
\begin{equation}\label{eqn:ineq}
\begin{aligned}
e^{\lambda_n^s - \hl_n^u} Z_n^\psi(t) + \rho_n(t) &\leq \Big(\big( e^{\lambda_n^s - \hl_n^u}  + L_1(1+L_1\bar\gamma)\xi\big) \kappa_n + L_1^2\xi \Big) t^\alpha \\
&\leq \big( e^{\lambda_n^s - \lambda_n^u+\zeta}  + L_1(1+L_1\bar\gamma)\xi  + L_1^2\xi^2 \big) \kappa_n t^\alpha,
\end{aligned}
\end{equation}
where the second inequality uses the fact that $1\leq \xi\kappa_n$ (from~\eqref{eqn:bd-b} and the fact that $\beta_n\geq 1$).  Using the fact that $\lambda_n^s - \lambda_n^u \geq -L$ and decreasing $\xi$ if necessary, the last quantity is at most $\leq e^{\lambda_n^s - \lambda_n^u + 2\zeta} \kappa_n t^\alpha$.

Using the inequality $(2+\alpha)\zeta \leq\delta$ and multiplying both sides of~\eqref{eqn:rec-k} by $e^{\alpha(\lambda_n^u - \zeta)}$ gives
\[
e^{\alpha(\lambda_n^u-\zeta)} \kappa_{n+1} \geq e^{\lambda_n^s - \lambda_n^u + 2\zeta}\kappa_n
\]
and so the quantities in~\eqref{eqn:ineq} are bounded above by
\[
e^{\alpha(\lambda_n^u-\zeta)} \kappa_{n+1} t^\alpha \leq e^{\alpha\hl_n^u} \kappa_{n+1} t^\alpha = Z_{n+1}^\psi(te^{\hl_n^u}),
\]
where the first inequality uses~\eqref{eqn:enu}.  This establishes~\eqref{eqn:kn}.

The last estimate we need before verifying the remaining hypotheses of Theorem~\ref{thm:HP1} in~\eqref{eqn:rn}--\eqref{eqn:gn} is an estimate on $\eps_n^\sigma$: the first inequality below uses~\eqref{eqn:L1}, and the second uses~\eqref{eqn:enf},~\eqref{eqn:bd-k}, and~\eqref{eqn:bd-b}.
\begin{equation}\label{eqn:ens}
\begin{aligned}
\eps_n^\sigma &\leq L_1 \kappa_n e^{-\alpha\hl_n^u} (\eps_n^f)^\alpha \tau_n^\alpha + L_1(1+\bar\gamma) \beta_n \tau_n^\alpha (1+L_1(\eps_n^f)(1+\bar\gamma))^\alpha \\
&\leq L_1^{1+\alpha} (3\xi)^\alpha \sigma_n + L_1(1+\bar\gamma) \xi \sigma_n (1+L_1(3\xi)(1+\bar\gamma))^\alpha \\
&\leq L_2 \xi^\alpha \sigma_n.
\end{aligned}
\end{equation}

Now we can verify the conditions.  To verify~\eqref{eqn:rn}, we estimate the right-hand side using~\eqref{eqn:enu},~\eqref{eqn:enf},~\eqref{eqn:bd-t}, and~\eqref{eqn:rec-r}:
\[
e^{\hl_n^u} r_n - \eps_n^f \tau_n \geq (e^{\lambda_n^u-\zeta} - 3\xi)r_n \geq e^{\lambda_n^u - \delta} r_n \geq r_{n+1},
\]
where again we decrease $\xi$ if necessary.  The condition~\eqref{eqn:taun} follows immediately from~\eqref{eqn:rec-t} and~\eqref{eqn:clns2}.  For~\eqref{eqn:thetan}, we use~\eqref{eqn:enu},~\eqref{eqn:ens}, and~\eqref{eqn:rec-s} to obtain
\[
e^{\lambda_n^s - \hl_n^u}\sigma_n + \eps_n^\sigma \leq (e^{\lambda_n^s - \lambda_n^u + \zeta} + L_2\xi^\alpha) \sigma_n \leq e^{\lambda_n^s - \lambda_n^u + \delta} \sigma_n \leq \sigma_{n+1},
\]
where as always we decrease $\xi$ if necessary.  (Note that this is only done finitely many times.)  Finally,~\eqref{eqn:gn} follows directly from~\eqref{eqn:bd-s}.  Thus we can apply Theorem~\ref{thm:HP1} to obtain well-definedness of the graph transform.  We get~\eqref{eqn:param-expansion} from~\eqref{eqn:dFn} and~\eqref{eqn:enchi}.  The inequalities~\eqref{eqn:nearby-orbit} and~\eqref{eqn:C-vnvn} come from~\eqref{eqn:attraction} and~\eqref{eqn:vnvn}, and~\eqref{eqn:param-contraction} follows from~\eqref{eqn:contraction} and~\eqref{eqn:clns2}.

\subsection{Proof of Theorem~\ref{thm:tau}}

Let $\delta=\min(\hat\chi^u - \bar\chi^u, \bar\chi^s - \hat\chi^s) >0$, and let $\xi,\bar\gamma>0$ be given by Theorem~\ref{thm:parameters}.  Let $\bar{r}>0$, $\bar\sigma,\bar\tau\geq 0$ be small enough and $\bar\kappa$ be large enough so that
\begin{equation}\label{eqn:smallenough}
e^{L'}\bar\beta \bar{r}^\alpha \leq \xi, \quad e^{L'}\bar\beta \leq \xi\bar\kappa, \quad \bar\tau \leq \bar{r}, \quad \bar\kappa\bar\tau^\alpha \leq \bar\sigma, \quad \bar\sigma + \bar\kappa\bar{r}^\alpha\leq\bar\gamma.
\end{equation}
Now let $\bar\kappa \leq \hat\kappa \leq \bar\kappa e^{\alpha N \bar\chi^u}$ be such that \eqref{eqn:smallenough} holds with $\bar\kappa$ replaced by $\hat\kappa$.  We will work with $\hat\kappa$ from now on.

Define $c_n>0$ by $c_0=1$ and $c_{n+1} = \min(e^{\lambda_n^e - \delta'} c_n, 1)$, where $\delta'=\frac\delta{2\alpha}$.  We claim that $c_n \geq e^{-M_n^u}$ for all $n$.  Indeed, if $m\in [0,n]$ is maximal such that $c_m=1$, then~\eqref{eqn:M-hyp} yields $c_n = e^{\sum_{k=m}^{n-1} (\lambda_k^e - \delta)} \geq e^{(n-m)\bar\chi^u - M_n^u} \geq e^{-M_n^u}$.

Similarly, define $\hat c_n>0$ by $\hat c_0 = e^{-N\bar\chi^u}$ and the same recursion $\hat c_{n+1} = \min(e^{\lambda_n^e - \delta'} \hat c_n, 1)$.  If $\hat c_m < 1$ for all $0\leq m \leq n$, then we have
\[
\hat c_n = e^{\sum_{k=0}^{n-1}(\lambda_k^e - \delta)} \hat c_0
\geq e^{n\bar\chi^u - M_n^u} e^{-N\bar\chi^u},
\]
whereas if $\hat c_m=1$ for some $m$ then we have $\hat c_n = c_n \geq e^{-M_n^u}$ for every $n\geq m$.  In particular we observe that $\hat c_N \geq e^{-M_N^u}$.

Now let $r_n = \bar{r} c_n$ and $\kappa_n = \bar\kappa \hat c_n^{-\alpha}$, so that in particular $\kappa_0 = \hat\kappa$.  We observe that $\kappa_n \leq \hat\kappa c_n^{-\alpha}$.  Using the fact that $\lambda_n^e \leq \lambda_n^u$ and $\alpha \lambda_n^e \leq (1+\alpha)\lambda_n^u - \lambda_n^s$, we see that the recursive relations~\eqref{eqn:rec-r} and~\eqref{eqn:rec-k} are satisfied.

Let $\tau_n$ and $\sigma_n$ be given by
\begin{equation}\label{eqn:tausigma}
\begin{aligned}
\tau_n &:= \bar\tau e^{-M_0^s} e^{\sum_{k=0}^{n-1} (\lambda_k^s + \delta')}, \\
\sigma_n &:= \hat\kappa c_n^{-\alpha} \tau_n^\alpha.
\end{aligned}
\end{equation}
Then~\eqref{eqn:rec-t} is satisfied immediately.  To show~\eqref{eqn:rec-s}, we observe that by the definitions of $c_n$ and $\tau_n$, we have
\begin{align*}
\frac{\sigma_{n+1}}{\sigma_n} &= \frac{c_{n+1}^{-\alpha} \tau_{n+1}^\alpha}{c_n^{-\alpha} \tau_n^\alpha} 
\geq e^{-\alpha(\lambda_n^e - \delta')} e^{\alpha(\lambda_n^s + \delta')} 
= e^{\alpha(\lambda_n^s - \lambda_n^e) + \delta},
\end{align*}
using the relation $\delta=2\alpha\delta'$.  Thus to prove \eqref{eqn:rec-s} it suffices to show that $\alpha(\lambda_n^s - \lambda_n^e) \geq \lambda_n^s - \lambda_n^u$.  If the right hand side is positive (there is a deficiency from domination), then by the definition of $\lambda_n^e$ we have $\lambda_n^e \leq \lambda_n^u + \frac 1\alpha (\lambda_n^u - \lambda_n^s)$, and so
\[
\alpha(\lambda_n^s - \lambda_n^e) \geq (1+\alpha)(\lambda_n^s - \lambda_n^u)
\geq \lambda_n^s - \lambda_n^u.
\]
On the other hand, if $\lambda_n^s \leq \lambda_n^u$, then $\lambda_n^e \leq \lambda_n^u$ and so
\[
\alpha(\lambda_n^s - \lambda_n^e) \geq \alpha(\lambda_n^s - \lambda_n^u) \geq \lambda_n^s - \lambda_n^u.
\]
This shows that \eqref{eqn:rec-s} holds.

We have the following estimates on $\tau_n$ and $\sigma_n$:
\begin{align}
\label{eqn:taunleq}
\tau_n &\leq \bar\tau e^{-M_n^u} e^{n\bar\chi^s}, \\
\label{eqn:sigmanleq}
\sigma_n &\leq \hat\kappa e^{\alpha M_n^u} \bar\tau^\alpha e^{-\alpha M_0^s} e^{\alpha \sum_{k=0}^{n-1} (\lambda_k^s + \delta)}
\leq \bar\sigma e^{\alpha n \bar\chi^s}.
\end{align}

To verify the bounds~\eqref{eqn:bd-r}--\eqref{eqn:bd-s}, we first observe that $\beta_n \leq \bar\beta c_{n+1}^{-\alpha}$.  To see this, let $m\in[0,n]$ be maximal such that $\beta_m \leq\bar\beta$ (noting that such an $m$ exists by the assumption that $\beta_0\leq\bar\beta$).  Then
\[
\beta_n \leq \bar\beta e^{-\alpha \sum_{k=m+1}^n \frac 1\alpha \log \frac{\beta_{k-1}}{\beta_k}}
\leq \bar\beta e^{-\alpha \sum_{k=m+1}^n \lambda_k^e}
\leq \bar\beta c_{m+1}^\alpha c_{n+1}^{-\alpha} \leq \bar\beta c_{n+1}^{-\alpha}.
\]
One consequence of this is the bound
\begin{equation}\label{eqn:thetageq}
\sin\theta_{n+1} \geq \beta_n^{-1} \geq \bar\beta^{-1} c_{n+1}^\alpha \geq \bar\beta^{-1} e^{-\alpha M_{n+1}^u},
\end{equation}
where the first inequality follows from~\ref{C3}, which lets us take $\bar\theta=\bar\beta^{-1}$ and use $\theta\geq \sin\theta$ to get the bound on $\theta_n$ in Part~\ref{Mus}  Another consequence is that
\[
\beta_n \leq \bar\beta e^{L'} c_n^{-\alpha},
\]
and so using~\eqref{eqn:smallenough} we have $\beta_n r_n^\alpha \leq e^{L'}\bar\beta c_n^{-\alpha} \bar{r}^\alpha c_n^\alpha = e^{L'}\bar\beta \bar{r}^\alpha \leq \xi$, and similarly $\beta_n \kappa_n^{-1} = e^{L'} \bar\beta\bar\kappa^{-1} \leq \xi$, which verifies~\eqref{eqn:bd-r} and~\eqref{eqn:bd-b}.  We see that \eqref{eqn:bd-s} follows from~\eqref{eqn:smallenough} since $\kappa_n r_n^\alpha \leq \hat\kappa c_n^{-\alpha} \bar{r} c_n^\alpha = \hat\kappa \bar{r}^\alpha$.

The bounds~\eqref{eqn:bd-r}--\eqref{eqn:bd-b} follow just as before, while~\eqref{eqn:bd-t} follows since from~\eqref{eqn:taunleq} since $\bar\tau\leq \bar{r}$.  The definition of $\sigma_n$ in~\eqref{eqn:tausigma} makes~\eqref{eqn:bd-k} immediate, and~\eqref{eqn:bd-s} follows from the final inequality in~\eqref{eqn:smallenough}.  Having verified all the conditions of Theorem~\ref{thm:parameters}, we  observe that Parts~\ref{Mus}--\ref{D-3} of Theorem~\ref{thm:tau} follow from Theorem~\ref{thm:parameters} and the inequality $c_n \geq e^{-M_n^u}$.

For Part~\ref{Mus3} of Theorem~\ref{thm:tau}, we will use Part~\ref{C-4} of Theorem~\ref{thm:parameters}.  We bound $\hat{r}_n$ by
\begin{align*}
\hat{r}_n &= e^{\sum_{k=0}^{n-1} (-\lambda_k^u + \delta)} r_n + 3\xi \sum_{k=0}^{n-1} e^{\sum_{j=0}^{k-1} (-\lambda_j^u + \delta)} \tau_k \\
&= e^{\sum_{k=0}^{n-1} (-\lambda_k^u + \delta)} \bar{r} c_n + 3\xi \sum_{k=0}^{n-1} e^{\sum_{j=0}^{k-1} (-\lambda_j^u + \delta)} \bar\tau e^{-M_0^s} e^{\sum_{j=0}^{k-1} (\lambda_j^s + \delta')} \\
&\leq e^{-n\bar\chi^u} e^{M_n^u} \bar{r} + 3\xi \sum_{k=0}^{n-1} e^{-k\bar\chi^u} e^{M_k^u} \bar\tau e^{-M_0^s} e^{k\bar\chi^s} e^{-M_k^u} e^{M_0^s},
\end{align*}
where the last line uses \eqref{eqn:M-hyp}, \eqref{eqn:Ms}, and the fact that $c_n\leq 1$.  Thus
\begin{equation}\label{eqn:hatrnsmall}
\hat{r}_n \leq e^{-n\bar\chi^u} e^{M_n^u} \bar{r}
+ 3\xi \sum_{k=0}^{n-1} e^{-k(\bar\chi^u - \bar\chi^s)} \bar\tau.
\end{equation}
Using \eqref{eqn:param-contraction} and \eqref{eqn:Ms}, we have
\[
\|\ph_n - \psi_n\|_{C^0} \leq e^{n\bar\chi^s} e^{-M_n^u} e^{M_0^s} \cdot 2(\bar\tau + \bar\gamma \hat r_n),
\]
and so by choosing $\xi$ small enough we can use \eqref{eqn:hatrnsmall} to guarantee that
\[
\|\ph_n - \psi_n\|_{C^0} \leq 
e^{n\bar\chi^s} e^{-M_n^u} e^{M_0^s} (3\bar\tau + 2e^{-n\bar\chi^u} e^{M_n^u} \bar{r}),
\]
which proves \eqref{eqn:chius}.

Finally, Part~\ref{hyp-times3} of Theorem~\ref{thm:tau} follows  directly from the following lemma, due to Pliss~\cite{vP72}; a proof may be found in~\cite[Lemma 11.2.6]{BP07}.

\begin{lemma}\label{lem:pliss}
Given $L \geq \chi > \hat\chi > 0$, let $\rho = (\chi-\hat\chi)/(L-\hat\chi)$.  Then, given any real numbers $\lambda_1,\dots,\lambda_N$ such that
\[
\sum_{j=1}^N \lambda_j \geq \chi N \qquad\text{and}\qquad \lambda_j\leq L \text{ for every } 1\leq j\leq N,
\]
there are $\ell\geq \rho N$ and $1<n_1<\cdots<n_\ell\leq N$ such that
\[
\sum_{j=n+1}^{n_i} \lambda_j \geq \hat\chi(n_i-n) \qquad\text{for every } 0\leq n<n_i \text{ and } i=1,\dots,\ell.
\]
\end{lemma}

\subsection{Proof of Theorem~\ref{thm:HP2}}


Theorem~\ref{thm:HP2} follows directly from Theorem~\ref{thm:tau} by setting $\bar\sigma=\bar\tau=0$.  To get the appropriate density observe that for every $\chi^u < \chi^e$ we have $\frac 1N \sum_{k=0}^{N-1} \lambda_k^e \geq \chi^u$ for all sufficiently large $N$, whence the density of hyperbolic times is at least $(\chi^u - \hat\chi^u)/(L-\hat\chi^u)$, and since $\chi^u<\chi^e$ was arbitrary, this suffices.


\subsection{Proof of Theorem~\ref{thm:HP5}}\label{sec:C1+pfs2}

Let $\Gamma = \{ n\leq 0 \mid M_n(\hat\chi) = 0\}$.  We show that $\Gamma$ has positive lower asymptotic density.  Indeed, by~\eqref{eqn:lambdabig2} and the hypothesis on $\hat\chi^u$, for every $\chi\in (\hat\chi^u,\chi^e)$ there exists $N_0 < 0$ such that for all $N\leq N_0$ we have $\sum_{N\leq k<0} \lambda_k^e \geq \chi|N|$.  Given such an $N$, let $m_0=m_0(N)$ be the smallest value of $m$ with the property that
\begin{equation}\label{eqn:msum1}
\sum_{m\leq k<0} (\lambda_k^e - \hat\chi^u) \leq N (\hat\chi^u - \chi).
\end{equation}
By the assumption on $N_0$, this inequality fails for all $m<N$, and so $m_0\geq N$.  Furthermore, since $\lambda_k^e \leq L$, the equality is true as long as $|m| \leq \hat\rho |N|$, where $\hat\rho = (\chi-\hat\chi)/(L-\hat\chi)$ as in Lemma~\ref{lem:pliss}.  It follows that $N\leq m_0 \leq \hat\rho N$.

Let $\Gamma_N$ be the set of effective hyperbolic times $n\in (m_0,0]$; that is, the set of $n$ such that
\begin{equation}\label{eqn:msum2}
\sum_{m\leq k<n} (\lambda_k^e -\hat\chi^u) \geq 0
\end{equation}
for all $m_0\leq m< n$.  We claim that
\begin{enumerate}
\item $\Gamma_N \subset \Gamma$; and
\item $\#\Gamma_N \geq \hat\rho^2 |N|$.
\end{enumerate}
For the first claim, observe that given $n\in \Gamma_N$, it suffices to prove~\eqref{eqn:msum2} for $m<m_0$.  We can set $m=m_0$ in~\eqref{eqn:msum1} and~\eqref{eqn:msum2} and take the difference of the two inequalities to obtain
\begin{equation}\label{eqn:msum3}
\sum_{n\leq k<0} (\lambda_k^e - \hat\chi^u) \leq N (\hat\chi^u - \chi).
\end{equation}
Furthermore, for $m<m_0$ we have
\begin{equation}\label{eqn:msum4}
\sum_{m\leq k<0} (\lambda_k^e -\hat\chi^u)> N(\hat\chi^u - \chi)
\end{equation}
by the definition of $m_0$.  Subtracting~\eqref{eqn:msum3} from~\eqref{eqn:msum4} gives
\[
\sum_{m\leq k<n} (\lambda_k^e-\hat\chi^u) > 0,
\]
and so $M_n(\hat\chi^u)=0$, so $n\in \Gamma$.

For the second claim, we observe that by Lemma~\ref{lem:pliss} we have $\#\Gamma_N \geq \hat\rho |m_0|$, and it follows that from the earlier estimates on $m_0$ that $\#\Gamma_N \geq \hat\rho^2 |N|$.  This holds for all $N\leq N_0$, and so $\Gamma$ has lower asymptotic density at least $\hat\rho^2$.  As $\chi$ approaches $\chi^e$, we have $\hat\rho^2 \to \left( \frac{\chi^e - \hat\chi^u}{L-\hat\chi^u}\right)^2$, which proves the claim regarding asymptotic density of $\Gamma$.

Now fix $\delta < \min (\hat\chi^u - \bar\chi^u, \hat\chi^g)$, and let $\bar\gamma, \xi$ be as in Theorem~\ref{thm:parameters}, and $\bar{r},\bar\kappa$ as in~\eqref{eqn:smallenough}; let $\bar\theta=\bar\beta^{-1}$.  We want to define a sequence $c_n$ that will satisfy the recursive relationship
\begin{equation}\label{eqn:cn-recurse}
c_{n+1} = \min(e^{\lambda_n^e - \delta} c_n, 1)
\end{equation}
and allow us to define $r_n,\kappa_n$ as in the proof of Theorem~\ref{thm:tau}.  To this end, we let $\Theta = \{m<0 \mid \beta_m \leq \bar\beta\}$, and note that $\Theta$ is infinite by the hypotheses of the theorem.  Given $m\in \Theta$, define $\{c_n^{(m)} \mid m\leq n\leq 0\}$ by $c_m^{(m)} = 1$ and by~\eqref{eqn:cn-recurse} for $m<n\leq 0$.

Given $n\in \Gamma$ and $m\in \Theta$ with $m\leq n$, we have as in the proof of Theorem~\ref{thm:tau} that $c_n^{(m)} = 1$.  In particular, together with the definition of $c_n^{(m)}$, this shows that if $n\leq 0$ is arbitrary, then given any $m_1 \leq n_1 < n$ and $m_2 \leq n_2 < n$ with $m_i \in \Theta$ and $n_i\in\Gamma$, we have $c_n^{(m_1)} = c_n^{(m_2)}$.  Thus we may define without ambiguity a sequence $c_n$ as follows: given $n$, pick any $n'\in \Gamma\cap (-\infty,n)$ and any $m\in \Theta\cap(-\infty,n']$, and let $c_n=c_n^{(m)}$.

Part \ref{thetangeq} of Theorem~\ref{thm:HP5} follows from the same argument as~\eqref{eqn:thetageq} in the proof of Theorem~\ref{thm:tau}.  Also as in that proof, we let $r_n = \bar{r} c_n$, $\kappa_n = \bar\kappa c_n^{-\alpha}$, and $\gamma_n = \bar\gamma$ for all $n$.  The arguments there show that~\eqref{eqn:hZ}--\eqref{eqn:gn} are satisfied, and so Parts \ref{inv-mfd} and \ref{B-exp} of Theorem~\ref{thm:HP5} follow from Parts \ref{inv} and \ref{inv-exp} of Theorem~\ref{thm:HP3}, noting the bound $c_n \geq e^{-M_n(\hat\chi^u)}$ from the proof of Theorem~\ref{thm:tau}. 

Part \ref{B-unique} of Theorem~\ref{thm:HP5} follows from Part \ref{unique} of Theorem~\ref{thm:HP3} once we verify~\eqref{eqn:backwards} using the criterion of asymptotic domination.  As in the proof of Theorem~\ref{thm:parameters}, for any fixed $\delta>0$ we can choose $\bar\gamma, \bar{r}, \bar\theta$ small enough and $\bar\kappa$ large enough that $\cl_n^s < \lambda_n^s + \delta$ and $\hl_n^u > \lambda_n^u - \delta$.  Choosing $\delta$ such that $2\delta < \hat\chi^g$, we see from~\eqref{eqn:asymp-domb} that
\[
\llim_{n\to-\infty} \frac 1{|n|} \sum_{k=n}^{-1} (\hl_k^u - \cl_k^s) > 0,
\]
which implies \eqref{eqn:backwards} because $\gamma_n=\bar\gamma$ is constant.

To complete the proof of Theorem~\ref{thm:HP5}, it remains only to show Part \ref{B-char}, but this follows directly from Part \ref{dyn-char-2} of Theorem~\ref{thm:HP3}.

\subsection{Proof of Propositions~\ref{prop:verifying} and~\ref{prop:verifying2}}

Fix $\bar\beta$ and let $\delta = \ud\{n \mid \beta_n > \bar\beta\}$.  Using~\ref{C4}, there exists $L'>0$ such that $\lambda_n^e = \lambda_n^u - \Delta_n$ if $\beta_n \leq \bar\beta$ and $\lambda_n^e \geq -L'$ otherwise.  Now we have
\[
\llim_{n\to\infty} \frac 1n \sum_{k=0}^{n-1} \lambda_k^e \geq \left( \frac 1n \sum_{k=0}^{n-1} (\lambda_k^u - \Delta_k) \right) - \delta L',
\]
and since $\delta L'$ can be made arbitrarily small by taking $\bar\beta$ large, we are done.

\section{Proofs of applications}\label{sec:appspfs}

\subsection{Proof of Theorem~\ref{thm:HP6}}
For Theorem~\ref{thm:HP6}, it suffices to apply Theorem~\ref{thm:HP5} using local coordinates around the backwards trajectory of $x$.  

\subsection{Proof of Theorem~\ref{thm:closing}}

Given parameters $r,\tau,\sigma,\kappa$, let $\CCC_n$ be defined as in~\eqref{eqn:C} for the decomposition $T_{f^n(x)} M = Df^n(E^u) \oplus Df^n(E^s)$.  Consider the collection of $u$-admissible manifolds
\[
\WWW_n^u(r,\tau,\sigma,\kappa) := \{ \exp_{f^n(x)}\graph\psi \mid \psi \in \CCC_n(r,\tau,\sigma,\kappa) \}.
\]
Define the set of $s$-admissible manifolds $\WWW_n^s$ similarly, with the roles of $s,u$ reversed.

Fix $\bar\chi^{s,u}$ such that $\hat\chi^s < \bar\chi^s<0<\bar\chi^u<\hat\chi^u$, and let $\bar\gamma,\bar{r},\bar\theta,\bar\sigma,\bar\tau,\bar\kappa>0$ be given by Theorem~\ref{thm:tau}.  Assume that the parameters are chosen so that the bounds in \eqref{eqn:param-bounds} hold when $\bar\kappa$ is replaced by $2\hat\kappa$, where $\hat\kappa = \bar\kappa e^{\alpha M^u}$.  Let $p_0$ be such that $p_0 \bar\chi^u \geq M^u \log 2$.

Using~\eqref{eqn:ceh-Mnu} and~\eqref{eqn:ceh-Mu} to verify~\eqref{eqn:M-hyp} and~\eqref{eqn:ceh-hMs}--\eqref{eqn:ceh-hMs2} to verify~\eqref{eqn:Ms}, we can apply Theorem~\ref{thm:tau} to show that for $p\geq p_0$, the map $f^p$ induces a well-defined graph transform
\[
\WWW_0^u(\bar{r},\bar\tau e^{-\hat{M}^s}, \bar\sigma e^{-\hat M^s}, 2\hat\kappa) 
\to
\WWW_p^u(\bar{r} e^{-M^u}, \bar\tau e^{-M^u} e^{p \bar\chi^s}, \bar\sigma e^{\alpha p \bar\chi^s},\hat\kappa).
\]
Let $\hat\tau = \frac 12 \bar\tau e^{-\hat{M}^s}$ and $\hat\sigma = \frac 12 \bar\sigma e^{-\hat{M}^s}$.  Then increasing $p_0$ if necessary, we have for $p\geq p_0$ that the graph transform induced by $f^p$ acts between
\[
\WWW_0^u(\bar{r},2\hat\tau,2\hat\sigma,2\hat\kappa)
\to \WWW_p^u(\bar{r} e^{-M^u}, \hat\tau,\hat\sigma,\hat\kappa).
\]
Let $\hat{r} = e^{-p_0 \bar\chi^u} e^{M^u}\bar{r} + \bar\tau$ and choose $\bar\tau$, $p_0$ such that $2\hat{r} \leq \bar{r} e^{-M^u}$.  Then by Part \ref{Mus3} of Theorem~\ref{thm:tau}, the graph transform induced by $f^p$ acts between
\[
\WWW_0^u(\hat{r},2\hat\tau,2\hat\sigma,2\hat\kappa)
\to \WWW_p^u(2\hat{r}, \hat\tau,\hat\sigma,\hat\kappa).
\]
Now we can choose $\eps>0$ such that under the conditions of the theorem, the map $\exp_x \circ \exp_{f^p(x)}^{-1}$ embeds $\WWW_p^u(2\hat{r}, \hat\tau,\hat\sigma,\hat\kappa)$ into $\WWW_0^u(\hat{r},2\hat\tau,2\hat\sigma,2\hat\kappa)$, and we can view the graph transform induced by $f^p$ as a self-map on $\WWW_0^u$.  By \eqref{eqn:chius}, this self-map is a contraction, and so iterating any $u$-admissible manifold under this transform yields a sequence of $u$-admissible manifolds converging to a fixed point of the transform -- that is, a $u$-admissible manifold $W^u$ near $x$ such that $f^p(W^u) \supset W^u$.

Apply the same argument to $s$-admissible manifolds we obtain a fixed point for the graph transform associated to $f^{-p}$ -- that is, an $s$-admissible manifold $W^s$ near $x$ such that $f^{-p}(W^s) \supset W^s$.  By the bounds that $\WWW_0^u$ and $\WWW_0^s$ impose on the geometry of $W^u$ and $W^s$, they have a unique intersection point $z$, which is the periodic point we seek.

\section{Derivation of classical Hadamard--Perron theorems}\label{sec:classical}

We state two classical Hadamard--Perron theorems that follow from Theorem~\ref{thm:HP3}.  The uniform version in Theorem~\ref{thm:uniform} is derived from~\cite[Theorem 6.2.8]{KH95}, while the non-uniform version in Theorem~\ref{thm:nonuniform} follows~\cite[Theorem 7.5.1]{BP07}.

\subsection{Uniform hyperbolicity}

Fix $r_0>0$ and let $\Omega = B^u(0,r_0)\times B^s(0,r_0) \subset \RR^d$, where $B^u$ and $B^s$ are the balls in $E^u = \RR^k$ and $E^s=\RR^{d-k}$, respectively.  Let $\mu,\lambda\in\RR$ be such that $\mu>\max(1,\lambda)$ and for each $n\leq 0$ let $f_n \colon \Omega \to \RR^d$ be a $C^1$ map such that for $(x,y)\in \RR^k \oplus \RR^{d-k}$
\[
f_n(x,y) = (A_n x + g_n(x,y), B_n y + h_n(x,y))
\]
for some linear maps $A_n\colon \RR^k\to\RR^k$ and $B_n\colon \RR^{d-k} \to \RR^{d-k}$ with $\|A_n^{-1}\| \leq \mu^{-1}$, $\|B_n\|\leq \lambda$ and $g_n(0) = 0$, $h_n(0)=0$.

\begin{theorem}\label{thm:uniform}
There exists $\gamma_0 = \gamma_0(\mu,\lambda) \in (0,1]$ such that for all $0 < \gamma < \gamma_0$ there exists $\delta_0 = \delta_0(\mu,\lambda,\gamma)$ such that the following is true.

If $\max(\|g_n\|_{C^1}, \|h_n\|_{C^1}) < \delta < \delta_0$ for all $n$, then there exist $\lambda' = \lambda'(\lambda,\gamma,\delta) < \mu' = \mu'(\mu,\gamma,\delta)$ such that $\lim_{\gamma,\delta\to 0} \lambda' = \lambda$, $\lim_{\gamma,\delta\to 0} \mu' = \mu$ and a unique family  $\{W_n^+\}_{n\in\ZZ}$ of $k$-dimensional $C^1$ manifolds
\[
W_n^+ = \{ (x,\ph_n^+(x)) \mid x\in \RR^k \} = \graph \ph_n^+
\]
where $\ph_n^+\colon B^u(r_0) \to B^s(r_0)$, $\sup_{n\leq 0} \|D\ph_n^+\| < \gamma$, for which the following properties hold.
\begin{enumerate}[label=(\roman{*})]
\item $f_n(W_n^+)\cap \Omega_{n+1} = W_{n+1}^+$.
\item $\|f_n(y) - f_n(z)\| > \mu'\|y - z\|$ for $y,z\in W_n^+$.
\item Let $\lambda' < \nu < \mu'$.  If $\|f_{n-L}^{-1} \circ \cdots \circ f_{n-1}^{-1}(z) \| < C\nu^{-L}\|z\|$ for all $L\geq 0$ and some $C>0$ then $z\in W_n^+$.
\end{enumerate}
\end{theorem}

\begin{remark}
The result in~\cite[Theorem 6.2.8]{KH95} covers stable manifolds as well; to get these one need only apply the above result to the sequence of inverse maps, placing similar requirements on the nonlinear parts of $f_n^{-1}$.
\end{remark}

\begin{proof}[Derivation of Theorem~\ref{thm:uniform} from Theorem~\ref{thm:HP3}]
Translating the hypotheses of Theorem~\ref{thm:uniform} into the notation of Theorems~\ref{thm:HP1} and~\ref{thm:HP3}, we have
\[
e^{\lambda_n^u} = \mu, \qquad e^{\lambda_n^s} = \lambda, \qquad \theta_n = \frac\pi2.
\]
Let $0 < \gamma_0\leq 1$ be such that
\[
\lambda(1+\gamma_0) < \mu,
\]
and given $0<\gamma<\gamma_0$, let $\delta_0$ be such that
\[
\max\left(1,\, \lambda + (1+\gamma^{-1})\delta_0\right) < \frac{\mu - \delta_0(1+\gamma)}{1+\gamma}.
\]
Now given $0<\delta<\delta_0$, let
\begin{align*}
\lambda' &:= \lambda + (1+\gamma^{-1})\delta, \\
\mu' &:= \frac{\mu - \delta(1+\gamma)}{1+\gamma}.
\end{align*}
If $\max(\|g_n\|_{C^1}, \|h_n\|_{C^1}) < \delta$, then we have $\hZ_n^f(t) <\delta$ for  all $t$, and so \eqref{eqn:rnbn} gives $\eps_n^f \leq \delta$.  Taking $\gamma_n=\gamma$ for all $n$, \eqref{eqn:hlnu}--\eqref{eqn:clns} give
\[
\eps_n^u \leq (1+\gamma)\delta,\quad \eps_n^s \leq (1+\gamma^{-1})\delta, \quad
\check\eps_n \leq (1+\gamma^{-1})\delta,\quad \eps_n^\chi\geq -\log(1+\gamma),
\]
from which we have
\[
\max\left(e^{\cl_n^s},e^{\hl_n^s}\right) \leq \lambda' < \mu' \leq e^{\chi_n} \leq e^{\hl_n^u}.
\]
In particular,~\eqref{eqn:backwards} is satisfied.  We see that~\eqref{eqn:rn}--\eqref{eqn:gn} are  satisfied if we take $\gamma_n = \gamma$ for all $n$ and if we take $r_n=r_0$.  

Thus it only remains to get moduli of continuity $Z_n^\psi$ satisfying \eqref{eqn:kn}, which we do via Proposition \ref{prop:checkmoduli}.  This requires checking that the sum in~\eqref{eqn:Znpsisum} converges when $n=0$.  In the notation of the present theorem, this sum becomes
\[
Z_0^\psi(t) = \sum_{k<0} \mu' \left(\frac \lambda{\mu'}\right)^{-(k+1)} \left( 1 + \frac {\lambda'}{\mu'} \gamma \right)
\hZ_k^f\left( t (\mu')^{-k} \right).
\]
Write $\xi = \lambda'/\mu' < 1$.  Then it suffices to check that the sum
\[
\sum_{m>0} \xi^m \hZ_{-m}^f (t(\mu')^m)
\]
converges and goes to $0$ as $t\to 0$.  Convergence is immediate for all $t$, because $\hZ_{-m}^f \leq \delta$.  For the limit, let $\alpha>0$ be arbitrary and take $M$ such that $\sum_{m>M} \xi^m < \alpha$.  Then take $\tau$ such that $\sum_{m=0}^M \hZ_{-m}^f(\tau(\mu')^m) < \alpha$.  It follows that for every $0<t<\tau$ we have
\[
\sum_{m>0} \xi^m \hZ_{-m}^f (t(\mu')^m) \leq \alpha\delta + \alpha.
\]
Since $\alpha$ was arbitrary this completes the proof: \eqref{eqn:Znpsisum} holds, hence Theorem~\ref{thm:HP3} applies, and the conclusions of Theorem~\ref{thm:HP3} imply the conclusions of Theorem~\ref{thm:uniform}.
\end{proof}

\subsection{Non-uniform hyperbolicity}

The classical non-uniform result can be found in~\cite[Theorem 7.5.1]{BP07}.  We give a version adapted to our notation and our convention of working with unstable manifolds rather than stable manifolds.  

In the non-uniform setting, one considers a sequence of diffeomorphisms and uses the Lyapunov metric, which has the effect that the rates of expansion and contraction are still uniform, as is the angle between the stable and unstable directions, but the amount of nonlinearity may grow.

Let $\Omega = B^u(0,r_0)\times B^s(0,r_0) \subset \RR^d$.  For each $n\leq 0$ let $f_n \colon \Omega \to \RR^d$ be a $C^{1+\alpha}$ map such that for $(x,y)\in \RR^k \oplus \RR^{d-k}$ we have
\[
f_n(v,w) = (A_n v + g_n(v,w), B_n w + h_n(v,w)),
\]
where $A_n\colon \RR^k\to \RR^k$ and $B_n\colon \RR^{d-k}\to\RR^{d-k}$ are linear maps and $g_n\colon \RR^d\to\RR^k$ and $h_n\colon\RR^d\to\RR^{d-k}$ are nonlinear maps defined for each $v\in B^s(r_0) \subset \RR^k$ and $w\in B^u(r_0) \subset \RR^{d-k}$, with the property that $g_n(0,0) = Dg_n(0,0) = h_n(0,0) = Dh_n(0,0) = 0$.

Given $n\leq 0$, write $F_n = f_{-1} \circ f_{-2} \circ \cdots \circ f_n$, and write $F_n^{-1}$ wherever the inverse is defined.  Let $\kappa$ be any number satisfying
\[
\max\{\lambda',\zeta^{1/\alpha}\} < \kappa < \mu',
\]
where the numbers $\lambda',\mu'$, and $\zeta$ satisfy
\[
\|A_n^{-1}\|^{-1} \geq \mu',\qquad
\|B_n\| \leq \lambda', \text{ where } \mu' > \max\{1,\lambda'\},
\]
as well as
\[
1 < \zeta < (\mu')^\alpha, \qquad 0 < \alpha \leq 1, \qquad C > 0
\]
such that
\[
\|Dg_n(v_1,w_1) - Dg_n(v_2,w_2)\| \leq C\zeta^{|n|} (\|v_1-v_2\| + \|w_1-w_2\|)^\alpha,
\]
and similarly for $h_n$.  

\begin{theorem}\label{thm:nonuniform}
There exist $D>0$ and $r_0>r>0$ and a map $\psi^u\colon B^u(r) \to \RR^{d-k}$ such that
\begin{enumerate}
\item $\psi^u$ is of class $C^{1+\alpha}$ and $\psi^u(0)=0$ and $D\psi^u(0)=0$;
\item $\|D\psi^u(v_1) - D\psi^u(v_2)\| \leq D\|v_1-v_2\|^\alpha$ for any $v_1,v_2\in B^u(r)$;
\item if $n\leq 0$ and $v\in B^u(r)$ then
\begin{align*}
F_n^{-1} (v,\psi^u(v)) &\in B^u(r) \times B^s(r), \\
\left\|F_n^{-1} (v,\psi^u(v)) \right\| &\leq D\kappa^n \|(v,\psi^u(v))\|;
\end{align*}
\item given $v\in B^u(r)$ and $w\in B^s(r)$, if there is a number $K>0$ such that
\[
F_n^{-1} (v,w) \in B^u(r)\times B^s(r), \qquad
\left\|F_n^{-1} (v,w) \right\| \leq K\kappa^n
\]
for every $n\leq 0$, then $w=\psi^u(v)$;
\item the numbers $D$ and $r$ depend only on the numbers $\lambda',\mu',\zeta,\alpha,\kappa$, and $C$.
\end{enumerate}
\end{theorem}

\begin{remark}
The result in~\cite[Theorem 7.5.1]{BP07} deals with stable manifolds rather than unstable manifolds.  In order for our approach to treat stable manifolds, we need to impose bounds on $f_n^{-1}$ rather than on $f_n$; ultimately this is due to the fact that we use Hadamard's approach (graph transform), while the proof in~\cite{BP07} uses Perron's approach (implicit function theorem).
\end{remark}

\begin{proof}[Derivation of Theorem~\ref{thm:nonuniform} from Theorem~\ref{thm:HP3}]
Choose $\gamma\in (0,1]$ such that
\[
(1+\gamma)\kappa < \mu'
\]
and define $C', C''$ by
\[
C' = C(1+\gamma)^{1+\alpha}, \qquad C'' = C'/\gamma.
\]
Let $\gamma_n = \gamma$ for all $n\leq 0$; then for any choice of $r_n>0$, we have
\begin{equation}\label{eqn:nuhZ}
(1+\gamma_n) \hZ_n^f(r_n(1+\gamma_n)) \leq C' \zeta^{|n|} r_n^\alpha.
\end{equation}
(Observe that $\zeta^{|n|} \to \infty$ as $n\to -\infty$.)  Let $r\in (0,r_0)$ be such that
\begin{equation}\label{eqn:kappabounds}
\lambda' + C''r^\alpha < \kappa < \frac{\mu' - C'r^\alpha}{1+\gamma},
\end{equation}
and define $r_n$ for $n< 0$ by
\begin{equation}\label{eqn:nurn}
r_n = \kappa^{n} r.
\end{equation}
Then since $\kappa^\alpha > \zeta$, we have $\zeta^{|n|} r_n^\alpha  < r^\alpha$ for all $n<0$, and in particular $\eps_n^f < \frac {C'}{1+\gamma} r^\alpha$.

Let $\chi_n < \hl_n^u < \lambda_n^u$ and $\cl_n^s, \hl_n^s > \lambda_n^s$ be as in~\eqref{eqn:hlnu}--\eqref{eqn:clns}.  Then~\eqref{eqn:nuhZ}--\eqref{eqn:nurn} imply that
\begin{equation}\label{eqn:nuhl}
\begin{aligned}
\max\left(e^{\cl_n^s},e^{\hl_n^s}\right) &\leq e^{\lambda_n^s} + C'' r^\alpha \leq \lambda' + C'' r^\alpha < \kappa, \\
e^{\chi_n} &= \frac{e^{\hl_n^u}}{1+\gamma_n} \geq \frac{e^{\lambda_n^u} - C' \zeta^n r_n^\alpha}{1+\gamma_n}
\geq \frac{\mu' - C' r^\alpha}{1+\gamma} > \kappa.
\end{aligned}
\end{equation}
This establishes~\eqref{eqn:rn}--\eqref{eqn:gn}, and~\eqref{eqn:backwards} follows since $\hl_k^u > \hl_k^s$ for all $k$.  Thus it only remains to find moduli of continuity $Z_n^\psi$ satisfying \eqref{eqn:kn}, which we again do via Proposition \ref{prop:checkmoduli}.  Once we have checked the convergence of the sum in~\eqref{eqn:Znpsisum}, we will be able to apply Theorem~\ref{thm:HP3} and derive the conclusions of Theorem~\ref{thm:nonuniform}.

The inequalities~\eqref{eqn:nuhl}, together with~\eqref{eqn:nuhZ} and~\eqref{eqn:nurn}, show that for $Z_0^\psi$ as in~\eqref{eqn:Znpsisum} we have
\[
Z_0^\psi(t) \leq \sum_{m<0} \kappa^{-1} \left(\frac \kappa{\lambda'}\right)^m C' \zeta^{-m} (t\kappa^m)^\alpha
\leq \kappa C' t^\alpha \sum_{m<0}\left(\frac \kappa{\lambda'}\right)^m.
\]
Thus Theorem~\ref{thm:HP3} proves the existence of a $C^1$ unstable manifold for the sequence $f_n$ with the dynamical properties claimed in Theorem~\ref{thm:nonuniform}.  Furthermore, it shows that $Z_0^\psi(t)$ is a modulus of continuity for $D\psi^u$, which shows that $\psi^u$ is $C^{1+\alpha}$ with H\"older constant $\kappa C' \sum_{m<0} (\kappa/\lambda')^m$, which completes the proof.
\end{proof}

\section{Relationship between non-uniform hyperbolicity and effective hyperbolicity}\label{sec:nuh}
We briefly discuss some differences between the notion of non-uniform hyperbolicity and the notion of effective hyperbolicity.  Note that these differences appear at the purely linear level and do not depend on how the different techniques deal with non-linear behaviour.

\subsection{(Non-uniform) hyperbolicity without effective hyperbolicity}

A sequence of germs may be non-uniformly hyperbolic but not effectively hyperbolic.  This can happen when there are multiple unstable directions which undergo expansion at different times: the notion of effective hyperbolicity used in this paper is not refined enough to detect this phenomenon.  For example, let $f_n\colon \RR^2\to \RR^2$ be defined by $f_n(x,y) = (3x,y/2)$ when $n$ is even, and $f_n(x,y) = (x/2,3y)$ when $n$ is odd.  Then $\lambda_n^u = -\log 2$ for every $n$ and hence $f_n$ is not effectively hyperbolic.  However, the sequence $f_n$ is non-uniformly hyperbolic with positive Lyapunov exponents $\frac 12 (\log 3 - \log 2)$ in all directions in $\RR^2$.

\subsection{Effective hyperbolicity without non-uniform hyperbolicity}

A sequence of germs may be effectively hyperbolic but not non-uniformly hyperbolic, i.e., without having slowly varying (tempered) constants, which are required for non-uniform hyperbolicity~\cite{BP07}.
For example, let $f_n\colon \RR\to \RR$ be defined by $f_n(x) = e^{\lambda_n}x$, where $\lambda_1=4$ and for $k\geq1$ we have
\[
\lambda_n = \begin{cases} 4 & 2^k \leq n < 2^k+2^{k-1}, \\
            -3 & 2^k + 2^{k-1} \leq n < 2^{k+1}. 
            \end{cases}
\]
Then $\llim_{n\to\infty} \frac 1n \sum_{k=0}^{n-1} \lambda_k = 1/2 > 0$, so the sequence is effectively hyperbolic, but if $M_n$ is any sequence of constants such that $\sum_{k=m}^n \lambda_k \geq (n-m)\chi - M_n$ for some $\chi \in (0,1/2)$ and every $0\leq m < n$, then the definition of $\lambda_n$ requires that
\[
M_{2^k} \geq \left(\sum_{j=2^k - 2^{k-2}}^{2^k - 1} \lambda_j\right) - 2^{k-2} \chi = 2^{k-2} (3-\chi).
\]
In particular, $\ulim_{n\to\infty} \frac 1n M_n > \frac 12 = \llim_{n\to\infty} \frac 1n \sum_{k=0}^{n-1} \lambda_k$, so any sequence of constants for non-uniform hyperbolicity must vary more quickly than the Lyapunov exponent.

The example described here is in some sense atypical -- the set of trajectories that are effectively hyperbolic but fail to be non-uniformly hyperbolic has measure zero with respect to any invariant measure.  Indeed, if an ergodic measure gives positive weight to the set of effectively hyperbolic trajectories, then it is a hyperbolic measure and the whole classical theory of non-uniform hyperbolicity applies.

We see from this that effective hyperbolicity is most useful when no a priori information about invariant measures is available.  This is the case, for example, when trying to construct SRB measures for dissipative systems.

\def\cprime{$'$} \def\cprime{$'$}
\providecommand{\bysame}{\leavevmode\hbox to3em{\hrulefill}\thinspace}
\providecommand{\MR}{\relax\ifhmode\unskip\space\fi MR }
\providecommand{\MRhref}[2]{%
\href{http://www.ams.org/mathscinet-getitem?mr=#1}{#2}
}
\providecommand{\href}[2]{#2}

\end{document}